\documentclass[bj, preprint]{imsart}
\usepackage[reqno]{amsmath}
\usepackage{ogonek}
\usepackage{amsthm,amsfonts,amssymb}
\usepackage[usenames,dvipsnames]{xcolor}
\usepackage{url}
\RequirePackage{natbib}
\RequirePackage[colorlinks,citecolor=blue,urlcolor=blue]{hyperref}

\arxiv{arXiv:1504.07717}

\startlocaldefs
\def\P{{\mathbb P}}
\def\E{{\mathbb E}}
\def\R{{\mathbb R}}
\newtheorem{theorem}{Theorem}[section]
\newtheorem{lemma}[theorem]{Lemma}
\theoremstyle{definition}

\theoremstyle{remark}
\newtheorem{remark}[theorem]{Remark}
\allowdisplaybreaks
\numberwithin{equation}{section}
\endlocaldefs

\begin{document}

\begin{frontmatter}

\title{Tail Asymptotics for the Extremes of Bivariate Gaussian Random Fields\thanks{Research supported in part by NSF grants DMS-1307470 and DMS-1309856.}}

\runtitle{Tail Asymptotics for Extremes of Bivariate Gaussian Fields}

\begin{aug}
\author{\fnms{Yuzhen} \snm{Zhou}\thanksref{a}\ead[label=e1]{yuzhenzhou@unl.edu}}
\and
\author{\fnms{Yimin} \snm{Xiao}\thanksref{b}\ead[label=e2] {xiao@stt.msu.edu}}
\address[a]{Department of Statistics, University of Nebraska-Lincoln, 340 Hardin Hall North Wing, Lincoln, NE 68583-0963 \printead{e1}}
\address[b]{Department of Statistics and Probability, Michigan State University, 619 Red Cedar Road, C413 Wells Hall, East Lansing, MI 48824-1027 
 \printead{e2}}

\runauthor{Y. Zhou and Y. Xiao}


\end{aug}

\begin{abstract}
Let $\{X(t)= (X_1(t),X_2(t))^T,\ t \in \mathbb{R}^N\}$ be an $\mathbb{R}^2$-valued continuous
locally stationary Gaussian random field with $\E[X(t)]=\mathbf{0}$.  For any compact sets
$A_1, A_2 \subset \R^N$, precise asymptotic behavior of the excursion probability
\[
\mathbb{P}\bigg(\max_{s\in A_1} X_1(s)>u,\, \max_{t\in A_2} X_2(t)>u\bigg),\ \
\text{ as }\ u \rightarrow \infty
\]
is investigated by applying the double sum method. The explicit results depend not only on the
smoothness parameters of the coordinate fields $X_1$ and $X_2$, but also on their maximum
correlation $\rho$.
\end{abstract}


\received{\smonth{4} \syear{2015}}


\end{frontmatter}

\section{Introduction}
\label{Introduction}

For a real-valued Gaussian random field $X = \{X(t),$ $ t\in T \}$, where $T$ is the
parameter set, defined on probability space $(\Omega, \mathcal{F}, \P)$,
the excursion probability $\mathbb{P}\{\sup_{t\in T}X(t)>u\}$ has been studied extensively.
Extending the seminal work of \cite{Pickands_1969}, \cite{Piterbarg_1996} developed
a systematic theory on asymptotics of the aforementioned excursion probability for
a broad class of Gaussian random fields. Their method, which is called the double sum method,
has been further extended by \cite{Chan_Lai_2006} to non-Gaussian random fields and, recently,
by \cite{DHJ_14} to a non-stationary Gaussian random field $\{X(s, t), (s, t)\in \R^2\}$
whose variance function attains its maximum on a finite number of disjoint line segments.
For smooth Gaussian random fields, more accurate approximation results have been established
by using integral and differential-geometric methods (see, e.g., \cite{Adler_2000},
\cite{Adler_Taylor_2007}, \cite{Azais_Wschebor_2009} and the references therein).
For Gaussian and asymptotically Gaussian random fields, the change of measure
method was developed by  \cite{Nardi_Siegmund_Yakir_2008} and \cite{Yakir_2013}.
Many of the results in the aforementioned references have found important applications
in statistics and other scientific areas. We refer to \cite{Adler_Taylor_Worsley_2012} and
\cite{Yakir_2013} for further information.

However, only a few authors have studied the excursion probability of multivariate random fields.
\cite{Piterbarg_Stamatovich_2005} and \cite{Debicki_Kosinski_2010} established large
deviation results for the excursion probability in multivariate case.
\cite{Anshin_2006} obtained precise asymptotics for a special class of nonstationary bivariate
Gaussian processes, under quite restrictive conditions. \cite{Hashorva_Ji_2014} recently
derived precise asymptotics for the excursion probability of a bivariate fractional Brownian
motion with constant cross correlation. The last two papers only consider multivariate
processes on the real line $\mathbb{R}$ with specific cross dependence structures.
\cite{Cheng_Xiao_2014} established a precise approximation to the excursion probability by using
the mean Euler characteristics of the excursion set for a broad class of smooth bivariate
Gaussian random fields on $\mathbb{R}^N$. In the present paper we investigate asymptotics
of the excursion probability of non-smooth bivariate Gaussian random fields on $\mathbb{R}^N$,
where the methods are totally different from the smooth case.

Our work is also motivated  by the recent increasing interest in using multivariate random fields
for modeling multivariate measurements obtained at spatial locations (see, e.g.,
\cite{Gelfand_Diggle_Fuentes_Guttorp_2010}, \cite{Wackernagel_2003}).
Several classes of multivariate spatial models have been introduced by
\cite{Gneiting_Kleiber_Schlather2010}, \cite{Apanasovich_Genton_Sun_2012} and \cite{Kleiber_Nychka_2012}.
We will show in Section 2 that the main results of this paper are applicable to bivariate
Gaussian random fields  with Mat\'{e}rn cross-covariances introduced by \cite{Gneiting_Kleiber_Schlather2010}.
Furthermore, we expect that the excursion probabilities considered in this paper
will have interesting statistical applications.

Let $\{X(t),t \in \mathbb{R}^N\}$ be an $\mathbb{R}^2$-valued (not-necessarily stationary) Gaussian
random field with $\E[X(t)]=\mathbf{0}$. We write $X(t)\triangleq(X_1(t),X_2(t))^T$ and define
\begin{equation}
r_{ij}(s,t):= \E[X_i(s)X_j(t)],\ i,j=1,2.
\end{equation}
Let $|t|:=\sqrt{\sum_{j=1}^Nt_j^2}$ be the $l^2$-norm of a vector $t\in \mathbb{R}^N$.
Throughout this paper, we impose the following assumptions.

\begin{itemize}
\label{condition for BGRF}
\item[i)] $r_{ii}(s,t)=1-c_i|t-s|^{\alpha_i}+o(|t-s|^{\alpha_i})$, where $\alpha_i\in (0,2)$
and $ c_i>0$ ($i=1,2$) are constants.
\item[ii)] $|r_{ii}(s,t)|<1$ for all $|t-s|>0$, $i=1,2$.
\item[iii)] $r_{12}(s,t)=r_{21}(s,t):=r(|t-s|)$. Namely, the cross correlation is isotropic.
\item[iv)] The function $r(\cdot): [0,\infty)\rightarrow \mathbb{R}$ attains maximum only
at zero with $r(0)=\rho\in (0,1)$, i.e., $|r(t)|<\rho$ for all $t>0$. Moreover, we assume
$r'(0)=0, r''(0) < 0$ and there exists $\eta>0$, for any $s\in [0,\eta]$, $r''(s)$ exists and continuous.
\end{itemize}

The cross correlation defined here is meaningful and common in spatial statistics
where it is usually assumed that the correlation decreases as the distance between
two observations increases (see, e.g., \cite{Gelfand_Diggle_Fuentes_Guttorp_2010},
\cite{Gneiting_Kleiber_Schlather2010}). We only assume that the cross correlation is
twice continuously differentiable around the area where the maximum correlation is attained,
which is a weaker assumption than that in  \cite{Cheng_Xiao_2014} who considered smooth
bivariate Gaussian fields.

For any compact sets $A_1, A_2 \subset \R^N$, we investigate the asymptotic
behavior of the following excursion probability
\begin{equation}
 \label{bivariate excursion probablity}
\mathbb{P}\bigg(\max_{s\in A_1} X_1(s)>u, \, \max_{t\in A_2} X_2(t)>u\bigg),\ \
\text{ as }\ u \rightarrow \infty.
\end{equation}
The main results of this paper are Theorems 2.1 and 2.2 below, which demonstrate that the
excursion probability (\ref{bivariate excursion probablity}) depends not only on the
smoothness parameters of the coordinate fields $X_1$ and $X_2$, but also on their maximum
correlation $\rho$. The proofs of our Theorems 2.1 and 2.2 will be based on the double sum method.
Compared with the earlier works of  \cite{Ladneva_Piterbarg_2000},  \cite{Anshin_2006} and
\cite{Hashorva_Ji_2014}, the main difficulty in the present paper is that the
correlation function of $X_1$ and $X_2$ attains its maximum over the set $D:=\{(s,s): \, s\in A_1
\cap A_2\}$ which may have different geometric configurations. Several non-trivial
modifications for carrying out the arguments in the double sum method have to be made.

This paper raises several open questions. First, the cases of $\alpha_1= 2$ or $\alpha_2= 2$ have not been
considered in this paper. The main difficulty is that, when $\alpha_1 = 2$, the sample functions of $X_1$ may 
either be differentiable or non-differentiable.   
In view of the method in this paper, the proof of Lemma \ref{uniformly conditional convergence} on 
the uniform convergence of finite dimensional distributions for bivariate process breaks down 
when $\alpha_1= 2$ or $\alpha_2= 2$. 
Studying the asymptotics of (\ref{bivariate excursion probablity}) when $\alpha_1= 2$ or/and $\alpha_2= 2$  
requires different methods for dealing with differentiable or non-differentiable cases. When both $X_1$ and $X_2$ have twice continuously 
differentiable sample functions, this problem has been studied by \cite{Cheng_Xiao_2014}. The authors plan to 
study the remaining cases  in their future work.
Second, it would be interesting to
study the excursion probabilities when $\{X(t),\,t  \in \mathbb{R}^N\}$ is anisotropic or non-stationary, 
or taking values in $\R^d$ with $d \ge 3$. In the last problem, 
the covariance and cross-covariance structures become more complicated.
We expect that the pairwise maximum cross correlations and the size (e.g.,
the Lebesgue measure) of the set  where all the pairwise cross correlations
attain their maximum values (if not empty) will play an important role.


The rest of the paper is organized as follows. Section \ref{sec_main results and discussion} states
the main theorems with some discussions and provides an application of the main theorems
to the bivariate Gaussian fields with Mat\'{e}rn cross-covariances introduced by
\cite{Gneiting_Kleiber_Schlather2010}.  We state the key lemmas and provide proofs of our main 
theorems in Section \ref{sec_proof of main results}. The proofs of the lemmas are given in  
Section \ref{sec_proof of lemmas}.

We end the introduction with some notation. 
For any $t\in \mathbb{R}^N$,  $|t|$ denotes its $l^2$-norm.  An integer vector
$\mathbf{k}\in \mathbb{Z}^N$ is written as $\mathbf{k}=(k_1,...,k_N)$. For $\mathbf{k}\in \mathbb{Z}^N$
and $T \in \mathbb{R}_+= [0, \infty)$, we define the cube $[\mathbf{k}T,(\mathbf{k}+1)T]:=
\prod_{i=1}^N [k_iT,(k_i+1)T]$. For any integer $n$, $mes_n(\cdot)$ denotes the $n$-dimensional
Lebesgue measure. An unspecified positive and finite constant will be denoted by $C_0$.
More specific constants are numbered by $C_1,C_2, \ldots.$

\section{Main Results and Discussions}
\label{sec_main results and discussion}

We recall the Pickands constant first (see, \cite{Pickands_1969, Piterbarg_1996, Dieker_Yakir_2014}).
Let $\chi = \{\chi(t),t\in \mathbb{R}^N\}$ be a (rescaled) fractional Brownian motion
with Hurst index $\alpha/2 \in (0, 1)$, which is a centered Gaussian field with
covariance function $\mathbb{E}[\chi(t)\chi(s)]=|t|^{\alpha}+|s|^{\alpha}-|t-s|^{\alpha}$.

As in \cite{Ladneva_Piterbarg_2000} and \cite{Anshin_2006}, we define for any
compact sets $\mathbb{S}, \mathbb{T} \subset \mathbb{R}^N$,
\begin{equation}
\label{H_alpha(S,T)}
H_\alpha(\mathbb{S},\mathbb{T}):=\int_0^\infty e^x\cdot\mathbb{P}\Big(\sup_{s\in \mathbb{S}}
\big(\chi(s)-|s|^\alpha \big)>x,\, \sup_{t\in \mathbb{T}}  \big(\chi(t)-|t|^\alpha\big)>x\Big)\,dx.
\end{equation}
Let $H_\alpha(\mathbb{T})=H_\alpha(\mathbb{T},\mathbb{T})$. Then, the Pickands constant is defined as
\begin{equation}
\label{Pickands constant}
H_\alpha:=\lim_{T\rightarrow \infty} \frac{H_\alpha([0,T]^N)}{T^N},
\end{equation}
which is positive and finite (cf. \cite{Piterbarg_1996}).

Before moving to the tail probability of extremes of a bivariate Gaussian random field,
let us consider the tail probability of a standard bivariate Gaussian vector $(\xi,\eta)$ with
correlation $\rho$. It is known that (see, e.g., \cite{Ladneva_Piterbarg_2000})
\begin{align*}
\mathbb{P}(\xi>u,\eta>u)=&\Psi(u,\rho)(1+o(1)),\ \text{as}\ u\rightarrow \infty,
\end{align*}
where
\begin{align*}
\Psi(u,\rho):=\frac{(1+\rho)^2}{2\pi u^2 \sqrt{1-\rho^2}}\exp\left(-\frac{u^2}{1+\rho}\right).
\end{align*}
The exponential part of the tail probability above is determined by the correlation $\rho$. As shown
by Theorems \ref{theorem_Jordan measurable sets}
and \ref{theorem: Jordan measurable sets with mes zero} below,
similar phenomenon also happens for the tail probability of double extremes of  $\{X(t),t \in \R^N\}$,
where the exponential part is determined by the maximum cross correlation of the coordinate fields $X_1$ and $X_2$.

We will study double extremes of $X$ on the domain $A_1\times A_2$ where $A_1,A_2$ are
bounded Jordan measurable sets in $\mathbb{R}^N$. That is, the boundaries
of $A_1$ and $A_2$ have $N$-dimensional Lebesgue measure $0$  (see, e.g.,
\cite{Piterbarg_1996}, p.$105$). We only consider the case when
$A_1\cap A_2\neq \emptyset$, in which the maximum cross correlation $\rho$ can be attained.

If $mes_N(A_1\cap A_2)\neq 0$, we have the following theorem.
\begin{theorem}
\label{theorem_Jordan measurable sets}
Let $\{X(t),t \in \R^N\}$ be a bivariate Gaussian random field that satisfies the
assumptions in Section \ref{Introduction}. If $mes_N(A_1\cap A_2)\neq 0$, then
as $u \rightarrow \infty$,
\begin{equation}
\label{double extremes asymptotics_Jordan measurable sets}
\begin{split}
&\mathbb{P}\bigg(\max_{s\in A_1} X_1(s)>u,\, \max_{t\in A_2} X_2(t)>u\bigg) \\
&=(2\pi)^{\frac{N}{2}}(-r''(0))^{-\frac{N}{2}}c_1^{\frac{N}{\alpha_1}}
c_2^{\frac{N}{\alpha_2}} mes_N(A_1\cap A_2)  H_{\alpha_1}H_{\alpha_2} \\
&\qquad \times (1+\rho)^{-N(\frac{2}{\alpha_1}+\frac{2}{\alpha_2}-1)}\,
u^{N(\frac{2}{\alpha_1}+\frac{2}{\alpha_2}-1)}\Psi(u,\rho)(1+o(1)).
\end{split}
\end{equation}
\end{theorem}

If $mes_N(A_1\cap A_2)=0$, the above theorem is not informative. We have not been able to
obtain a general explicit formula. Instead, we consider the special cases
\begin{equation}\label{Eq:intM}
A_1=A_{1,M}\times \prod_{j=M+1}^N[S_j,T_j]\ \hbox{ and }\ A_2=A_{2,M}\times \prod_{M+1}^N[T_j,R_j],
\end{equation}
where $A_{1,M}$ and $A_{2,M}$ are $M$ dimensional Jordan sets with $mes_M(A_{1,M}\cap A_{2,M})
\neq 0$ and $S_j\leq T_j\leq R_j,\, j=M+1, \ldots, N, \, 0\leq M\leq N-1$. For simplicity
of notation, let $mes_0(\cdot)\equiv 1$. Our next theorem shows
that the excursion probability is smaller than that in
(\ref{double extremes asymptotics_Jordan measurable sets}) by a factor of $u^{M-N}$.

\begin{theorem}
\label{theorem: Jordan measurable sets with mes zero}
Let $\{X(t),t \in \R^N\}$ be a bivariate Gaussian random field that satisfies the assumptions in
Section \ref{Introduction}, and let $A_1,A_2$ be as in (\ref{Eq:intM}) such that $mes_M(A_{1,M}\cap A_{2,M})>0$.
Then as $u \to \infty$,
\begin{equation}
\label{double extremes asymptotics_Jordan measurable sets_mes zero}
\begin{split}
&\mathbb{P}\bigg(\max_{s\in A_1} X_1(s)>u,\, \max_{t\in A_2} X_2(t)>u\bigg) \\
&=(2\pi)^{\frac{M}{2}}(-r''(0))^{-\frac{2N-M}{2}}c_1^{\frac{N}{\alpha_1}}c_2^{\frac{N}
{\alpha_2}} H_{\alpha_1}H_{\alpha_2}  mes_M(A_{1,M}\cap A_{2,M})\\
&\qquad \times   (1+\rho)^{2N-M-\frac{2N}{\alpha_1}-\frac{2N}{\alpha_2}}\,
u^{M+N(\frac{2}{\alpha_1}+\frac{2}{\alpha_2}-2)}\Psi(u,\rho)(1+o(1)).
\end{split}
\end{equation}
\end{theorem}

\setcounter{theorem}{0}
\begin{remark} The following are some additional remarks about Theorems
\ref{theorem_Jordan measurable sets} and
\ref{theorem: Jordan measurable sets with mes zero}.
\begin{itemize}
\item The excursion probability in (\ref{bivariate excursion probablity})
depends on the region where the
maximum cross correlation is attained. In our setting, the maximum cross
correlation $\rho$ is attained on $D:=\{(s,s)\ |\ s\in A_1\cap A_2\}$.

\item For Theorem \ref{theorem: Jordan measurable sets with mes zero}, let us
consider the extreme case when $M=0$, i.e., $A_1\cap A_2=\{(T_1,...,T_N)\}$.
The exponential part still reaches $-\frac{u^2}{1+\rho}$, although the
maximum cross correlation $\rho$ is attained at a single point.

\item To compare our results with \cite{Anshin_2006}, we consider a centered
Gaussian process $\{X(t)=(X_1(t),X_2(t)),t \in \mathbb{R}\}$ and $A_1=A_2=[0,T]$. In our
setting, the cross correlation attains its maximum on the line $D=\{(s,s)\ |\ s\in [0,T]\}$,
while in \cite{Anshin_2006} it only attains at a unique point in $[0,T]\times[0,T]$ because
of the assumption $\mathbf{C2}$. This is the reason why
the power of $u$ in our settings is $\frac{2}{\alpha_1}+\frac{2}{\alpha_2}-3$
instead of $\frac{2}{\alpha_1}+\frac{2}{\alpha_2}-4$ in \cite{Anshin_2006}.
\item Even though Theorem \ref{theorem: Jordan measurable sets with mes zero}
only deals with a special case of  $A_1$, $A_2$ with $mes_N(A_1\cap A_2)=0$,
its method of proof can be applied to more general cases provided some information
on $A_1$ and $A_2$ is specified. The key step is to reevaluate the infinite
series in Lemma \ref{lem_Riemann Sum-2}.
\end{itemize}
\end{remark}
\setcounter{theorem}{2}

We end this section with an application of Theorems \ref{theorem_Jordan measurable sets}
and \ref{theorem: Jordan measurable sets with mes zero} to bivariate Gaussian random fields
with the Mat\'{e}rn correlation functions introduced by \cite{Gneiting_Kleiber_Schlather2010}.

The Mat\'{e}rn correlation function $M(h|\nu,a)$, where $a>0,\nu>0$ are scale and
smoothness parameters, is widely used
to model covariance structures in spatial statistics. It is defined as
\begin{equation}
M(h|\nu,a):=\frac{2^{1-\nu}}{\Gamma(\nu)}(a|h|)^\nu K_\nu(a|h|),
\end{equation}
where $K_\nu$ is a modified Bessel function of the second kind.
In \cite{Gneiting_Kleiber_Schlather2010}, the authors introduce the full bivariate
Mat\'{e}rn field $X(s)=(X_1(s),X_2(s))$, i.e., an $\mathbb{R}^2$-valued Gaussian
random field on $\mathbb{R}^N$ with zero mean and matrix-valued covariance functions:
\begin{equation}
C(h)=\left(\begin{array}{ll}
C_{11}(h)&C_{12}(h)\\
C_{21}(h)&C_{22}(h)
\end{array}
\right),
\end{equation}
where $C_{ij}(h):=\E[X_i(s+h)X_j(s)]$ are specified by
\begin{eqnarray}
C_{11}(h)&=&\sigma_1^2M(h|\nu_1,a_1),\\
C_{22}(h)&=&\sigma_2^2M(h|\nu_2,a_2),\\
C_{12}(h)&=&C_{21}(h)=\rho\sigma_1\sigma_2M(h|\nu_{12},a_{12}).\label{Matern Cross Covariance}
\end{eqnarray}

According to \cite{Gneiting_Kleiber_Schlather2010}, the above model
is valid if and only if
\begin{equation}
\label{validy cond}
\begin{split}
\rho^2 &\leq \frac{\Gamma(\nu_1+N/2)\Gamma(\nu_2+N/2)}{\Gamma(\nu_1)\Gamma(\nu_2)}
\frac{\Gamma(\nu_{12})^2}{\Gamma(\nu_{12}+N/2)^2}
\frac{a_1^{2\nu_1}a_2^{2\nu_2}}{a_{12}^{4\nu_{12}}}\\
& \qquad \  \ \times
\inf_{t\geq 0}\frac{(a_{12}^2+t^2)^{2\nu_{12}+N}}{(a_1^2+t^2)^{\nu_1+N/2}(a_2^2+t^2)^{\nu_2+N/2}}.
\end{split}
\end{equation}

Especially, when $a_1=a_2=a_{12}$, condition (\ref{validy cond}) is reduced to
\begin{equation}
\rho^2\leq \frac{\Gamma(\nu_1+N/2)\Gamma(\nu_2+N/2)}{\Gamma(\nu_1)\Gamma(\nu_2)}
\frac{\Gamma(\nu_{12})^2}{\Gamma(\nu_{12}+N/2)^2},
\end{equation}
in which case the choice of $\rho$ is fairly flexible.

Here we focus on a standardized bivariate Mat\'{e}rn field, that is, we assume
$\sigma_1=\sigma_2=1$, $a_1=a_2=a_{12}=1$ and $\rho>0$. Moreover, we assume $\nu_1,\nu_2\in (0,1)$ and $\nu_{12}> 1$.
In this case, the bivariate Mat\'{e}rn field $\{X(t), \, t \in \R^N\}$
satisfies the assumptions in Section \ref{Introduction}.

Indeed, Assumption i) in Section \ref{Introduction} is satisfied since 
\begin{align*}
M(h|\nu_i,a)=1-c_i|t|^{2\nu_i}+o(|t|^{2\nu_i}),
\end{align*}
where $c_i=\frac{\Gamma(1-\nu_i)}{2^{2\nu_i}\Gamma(1+\nu_i)},\ i=1,2$ (see, e.g., \cite{Stein_1999Interpo}, p. 32).
Assumption ii) holds immediately if we use the following integral representation of $M(h|\nu,a)$
(see, e.g., \cite{Abramowitz_Stegun_1972}, Section $9.6$)
\begin{align}
\label{integral form for Matern correlation}
M(h|\nu,a)=\frac{2\Gamma(\nu+1/2)}{\sqrt{\pi}\Gamma(\nu)}
\int_0^\infty \frac{\cos(a|h|r)}{(1+r^2)^{\nu+1/2}}\, dr.
\end{align}
Assumption iii) holds by the definition of cross correlation in
\eqref{Matern Cross Covariance}. For Assumption iv),
we only need to check the smoothness of $M(h|\nu,a)$. By another
integral representation of $M(h|\nu,a)$ (see, e.g., \cite{Abramowitz_Stegun_1972}, Section $9.6$), i.e.,
\begin{align*}
M(h|\nu,a)=\frac{2^{1-2\nu}(a|h|)^{2\nu}}{\Gamma(\nu+1/2)\Gamma(\nu)}\int_1^\infty e^{-a|h|r}(r^2-1)^{\nu-1/2}\,dr,
\end{align*}
one can verify that $M(h|\nu,a)$ is infinitely differentiable when $|h|\neq 0$.
Meanwhile, $M''(0|\nu,a)$ exists and is continuous when $\nu>1$ which can be
proven by taking twice derivatives to the integral representation in
\eqref{integral form for Matern correlation} w.r.t. $|h|$. So Assumption iv)
holds.


Applying Theorem \ref{theorem_Jordan measurable sets} to the double excursion probability of
$X(s)$ over $[0,1]^N$, we have
\begin{eqnarray*}
\label{double extremes for matern field}
&&\mathbb{P}\Big(\max_{s\in [0,1]^N} X_1(s)>u,\, \max_{t\in [0,1]^N} X_2(t)>u\Big)\nonumber\\
&&=(2\pi)^{\frac{N}{2}}(-C_{12}''(0))^{-\frac{N}{2}}c_1^{\frac{N}{2\nu_1}}c_2^{\frac{N}{2\nu_2}}(1+\rho)^{-N(\frac{1}
{\nu_1}+\frac{1}{\nu_2}-1)} H_{2\nu_1}H_{2\nu_2} \nonumber\\
&&\quad \times  u^{N(\frac{1}{\nu_1}+\frac{1}{\nu_2}-1)}\Psi(u,\rho)(1+o(1)),\ \ \text{ as }\, u \rightarrow \infty.
\end{eqnarray*}

Secondly, when the two measurements are observed on two regions
which only share part of boundaries, we use Theorem
\ref{theorem: Jordan measurable sets with mes zero} to obtain
the excursion probability. For example, if $X_1(s)$ are
observed on the region $[0,1]^N$ and $X_2(s)$ on $[0,1]^{N-1}\times [1,2]$,
then as $u\rightarrow \infty$,
\begin{eqnarray*}
\label{double extremes for matern field 2}
&&\mathbb{P}\bigg(\max_{s\in [0,1]^N} X_1(s)>u,\, \max_{t\in [0,1]^{N-1}\times [1,2]} X_2(t)>u\bigg)\nonumber\\
&&=(2\pi)^{\frac{N-1}{2}}(-C_{12}''(0))^{-\frac{N+1}{2}}c_1^{\frac{N}{2\nu_1}}c_2^{\frac{N}{2\nu_2}}
(1+\rho)^{1-N(\frac{1}{\nu_1}+\frac{1}{\nu_2}-1)} H_{2\nu_1}H_{2\nu_2} \nonumber\\
&&\qquad \times u^{N(\frac{1}{\nu_1}+\frac{1}{\nu_2}-1)-1}\Psi(u,\rho)(1+o(1)).
\end{eqnarray*}

\section{Proofs of the main results}
\label{sec_proof of main results}

The proofs of Theorems \ref{theorem_Jordan measurable sets} and
\ref{theorem: Jordan measurable sets with mes zero} are based on
the double sum method (\cite{Piterbarg_1996}) and the work of
\cite{Ladneva_Piterbarg_2000}. Since the latter deals with the
tail probability $\mathbb{P}(\max_{t\in[T_1,T_2]}X(t)>u,\,
\max_{t\in[T_3,T_4]}X(t)>u)$ of a univariate
 Gaussian process $\{X(t),t\in \mathbb{R}\}$, their method
is not sufficient for carrying out the double sum method for a
bivariate random field.


Lemmas \ref{main lemma} and \ref{double double local extremes lemma} below extend Lemma 
$1$ and Lemma $9$ in \cite{Ladneva_Piterbarg_2000} to
the bivariate random field $\{(X_1(t), X_2(t)),\, t \in \mathbb{R}^N\}$. Moreover,
we have strengthened the conclusions by showing that the convergence is uniform in certain sense.
This will be useful for dealing with sums of local approximations around the
regions where the maximum cross correlation is attained. The details will be illustrated
in the proof of Theorem \ref{theorem_Jordan measurable sets} (see, e.g.,
\eqref{upper bound around maximum 2}, \eqref{lower bound around maximum 2}).
In the following lemmas, $\{X(t), \, t \in \R^N\}$  is a bivariate Gaussian random field as
defined in Section \ref{Introduction}.

\begin{lemma}
\label{main lemma}
Let $s_u$ and $t_u$ be two $\mathbb{R}^N$-valued functions of $u$ and
let $\tau_u:=t_u-s_u$. For any compact sets  $\mathbb{S}$ and
$\mathbb{T}$ in $\mathbb{R}^N$, we have
\begin{equation}
\label{local double extremes asymptotics}
\begin{split}
&\mathbb{P}\bigg(\max_{s\in s_u+u^{-2/\alpha_1}\mathbb{S}} X_1(s)>u, \max_{t\in t_u+u^{-2/\alpha_2}\mathbb{T}} X_2(t)>u\bigg)\\
&=\frac{(1+\rho)^2}{2\pi \sqrt{1-\rho^2}}H_{\alpha_1}\left(\frac{c_1^{1/\alpha_1}\mathbb{S}}{(1+\rho)^{\frac{2}{\alpha_1}}}\right)H_{\alpha_2}\left(\frac{c_2^{1/\alpha_2}\mathbb{T}}{(1+\rho)^{\frac{2}{\alpha_2}}}\right)\\
& \qquad \times u^{-2} \exp\left(-\frac{u^2}{1+r(|\tau_u|)}\right)\, (1+o(1)),
\end{split}
\end{equation}
where $o(1)\rightarrow 0$ uniformly w.r.t. $\tau_u$ satisfying $|\tau_u|\leq C_0\sqrt{\log u}/u$ as $u \rightarrow \infty$.
\end{lemma}

\begin{lemma}
\label{double double local extremes lemma}
Let $s_u,\, t_u$ and $\tau_u$ be the same as in Lemma \ref{main lemma}.
For all $T>0$, $\mathbf{m},\mathbf{n}\in \mathbb{Z}^N$, we have
\begin{align}
\label{double double local extremes asymptotics}
&\mathbb{P}\bigg(\max_{s\in s_u+u^{-2/\alpha_1}[0,T]^N} X_1(s)>u, \, \max_{t\in t_u+u^{-2/\alpha_2}[0,T]^N} X_2(t)>u, \nonumber\\
&\qquad  \  \max_{s\in s_u+u^{-2/\alpha_1}[\mathbf{m}T,(\mathbf{m}+1)T]}X_1(s)>u,\, \max_{t\in t_u+u^{-2/\alpha_2}[\mathbf{n}T,(\mathbf{n}+1)T]} X_2(t)>u\bigg)\nonumber\\
&=\frac{(1+\rho)^2}{2\pi \sqrt{1-\rho^2}\, u^{2}}\, e^{-\frac{u^2}{1+r(|\tau_u|)}}\,H_{\alpha_1}\bigg(\frac{c_1^{1/\alpha_1}[0,T]^N}{(1+\rho)^{\frac{2}{\alpha_1}}},\frac{c_1^{1/\alpha_1}[\mathbf{m}T,(\mathbf{m}+1)T]}{(1+\rho)^{\frac{2}{\alpha_1}}}\bigg)\nonumber\\
&\qquad  \qquad \times\, H_{\alpha_2}\bigg(\frac{c_2^{1/\alpha_2}[0,T]^N}{(1+\rho)^{\frac{2}{\alpha_2}}},\frac{c_2^{1/\alpha_2}[\mathbf{n}T,(\mathbf{n}+1)T]}{(1+\rho)^{\frac{2}{\alpha_2}}}\bigg)\,\big(1+o(1)\big),
\end{align}
where $H_\alpha(\cdot,\cdot)$ is defined in \eqref{H_alpha(S,T)} and $o(1)\rightarrow 0$  uniformly for all $s_u$ and $t_u$ that satisfy  $|\tau_u|\leq C_0\sqrt{\log u}/u$  as $u \rightarrow \infty$.
\end{lemma}

Proofs of Lemmas \ref{main lemma} and \ref{double double local extremes lemma} will be given in Section 4. 
Now we proceed to prove our main theorems.
\begin{proof}[Proof of Theorem \ref{theorem_Jordan measurable sets}]
Let $\Pi=A_1\times A_2,\ \delta(u)=C\sqrt{\log u}/u$,
where $C$ is a constant whose value will be determined later. Let
\begin{align}
\label{domain D}
\mathcal{D}= \big\{(s,t)\in \Pi:  |t-s|\leq \delta(u) \big\}.
\end{align}
Since
\begin{align}
&\mathbb{P}\bigg(\bigcup_{(s,t)\in \mathcal{D}} \{X_1(s)>u, X_2(t)>u)\}\bigg)
\leq \mathbb{P}\bigg(\max_{s\in A_1} X_1(s)>u,\, \max_{t\in A_2} X_2(t)>u\bigg)\nonumber\\
& \leq \mathbb{P}\bigg(\bigcup_{(s,t)\in \mathcal{D}} \{X_1(s)>u,  X_2(t)>u)\}\bigg)
+\mathbb{P}\bigg(\bigcup_{(s,t)\in\Pi\setminus \mathcal{D}} \{X_1(s)>u,  X_2(t)>u)\}\bigg),\nonumber
\end{align}
it is sufficient to prove that, by choosing appropriate constant $C$, we have
\begin{equation}
\label{around maximum}
\begin{split}
&\mathbb{P}\bigg(\bigcup_{(s,t)\in \mathcal{D}} \{X_1(s)>u, X_2(t)>u)\}\bigg) \\
&=(2\pi)^{\frac{N}{2}}(-r''(0))^{-\frac{N}{2}}c_1^{\frac{N}{\alpha_1}}c_2^{\frac{N}{\alpha_2}}(1+\rho)^{-N(\frac{2}{\alpha_1}+\frac{2}{\alpha_2}-1)}
mes_N(A_1\cap A_2) \\
&\qquad \times  H_{\alpha_1}H_{\alpha_2} \, u^{N(\frac{2}{\alpha_1}+\frac{2}{\alpha_2}-1)}\Psi(u,\rho)(1+o(1)),\ \ \text{ as }   u \rightarrow \infty
\end{split}
\end{equation}
and
\begin{eqnarray}
\label{off maximum}
\lim_{u \rightarrow \infty}\frac{\mathbb{P}\left(\bigcup_{(s,t)\in\Pi\setminus \mathcal{D}}
\{X_1(s)>u, X_2(t)>u)\}\right)}{\mathbb{P}\left(\bigcup_{(s,t)\in \mathcal{D}} \{X_1(s)>u, X_2(t)>u)\}\right)}=0.
\end{eqnarray}

We prove \eqref{around maximum} first. For any fixed $T>0$ and
$i=1,2$, let $d_i(u)=Tu^{-\frac{2}{\alpha_i}}$ and, for any $\mathbf{k}
=(k_1,\ldots,k_N)\in \mathbb{Z}^N$, define
\begin{eqnarray}
\Delta^{(i)}_\mathbf{k} \triangleq \prod_{j=1}^N[k_jd_i(u),(k_j+1)d_i(u)]
= [\mathbf{k}d_i(u),(\mathbf{k}+1)d_i(u)].
\end{eqnarray}
Let
\begin{equation}
\label{mathcal C}
\mathcal{C}=\{(\mathbf{k} ,\mathbf{l}): \Delta^{(1)}_\mathbf{k}\times \Delta^{(2)}_\mathbf{l} \cap \mathcal{D}\neq \emptyset \}\ \text{ and }\
\mathcal{C}^\circ =\{(\mathbf{k},\mathbf{l}):\Delta^{(1)}_\mathbf{k}\times \Delta^{(2)}_\mathbf{l} \subseteq \mathcal{D} \}.
\end{equation}
It is easy to see that
$$\bigcup_{(\mathbf{k},\mathbf{l})\in \mathcal{C}^\circ}\Delta^{(1)}_\mathbf{k}\times \Delta^{(2)}_\mathbf{l}\subseteq \mathcal{D} \subseteq \bigcup_{(\mathbf{k},\mathbf{l})\in \mathcal{C}}\Delta^{(1)}_\mathbf{k}\times \Delta^{(2)}_\mathbf{l}. $$
Thus the LHS of \eqref{around maximum} is bounded above by
\begin{equation}
\label{upper bound around maximum}
\begin{split}
&\mathbb{P}\bigg(\bigcup_{(s,t)\in \mathcal{D}} \{X_1(s)>u,X_2(t)>u)\}\bigg)\\
&\leq \sum_{(\mathbf{k},\mathbf{l})\in \mathcal{C}}\mathbb{P}\bigg(\max_{s\in \Delta^{(1)}_\mathbf{k}} X_1(s)>u, \max_{t\in \Delta^{(2)}_\mathbf{l} } X_2(t)>u\bigg)\\
&=  \sum_{(\mathbf{k},\mathbf{l})\in \mathcal{C}}\mathbb{P}\bigg(\max_{s\in \mathbf{k}d_1(u)+\Delta^{(1)}_0} X_1(s)>u, \max_{t\in \mathbf{l}d_2(u)+\Delta^{(2)}_0 } X_2(t)>u\bigg).
\end{split}
\end{equation}
Let
\begin{equation}
\begin{split}
\tau_{\mathbf{k}\mathbf{l}}&:=\mathbf{l}d_2(u)-\mathbf{k}d_1(u)\\
&\ =(l_1d_2(u)-k_1d_1(u),...,l_Nd_2(u)-k_Nd_1(u)).
\label{def_tau_kl}
\end{split}
\end{equation}
For $(\mathbf{k},\mathbf{l})\in \mathcal{C}$, $|\tau_{\mathbf{k}\mathbf{l}}|\leq \delta(u)
+\sqrt{N}(d_1(u)+d_2(u))\leq 2\delta(u)$ for all $u$ large enough, since $d_1(u)=o(\delta(u))$ and $d_2(u)=o(\delta(u))$, as $u\rightarrow \infty$.
By applying Lemma \ref{main lemma} to the RHS of \eqref{upper bound around maximum},  we obtain
\begin{equation}
\label{upper bound around maximum 2}
\begin{split}
&\mathbb{P}\bigg(\bigcup_{(s,t)\in \mathcal{D}} \{X_1(s)>u,X_2(t)>u)\}\bigg)\\
&\leq \frac{(1+\rho)^2(1+\gamma(u))}{2\pi \sqrt{1-\rho^2}\, u^{2}}H_{\alpha_1}\left(\frac{c_1^{1/\alpha_1}[0,T]^N}{(1+\rho)^{\frac{2}{\alpha_1}}}\right)H_{\alpha_2}\left(\frac{c_2^{1/\alpha_2}[0,T]^N}{(1+\rho)^{\frac{2}{\alpha_2}}}\right)\\
&\qquad \ \times \sum_{(\mathbf{k},\mathbf{l})\in \mathcal{C}}\exp\left(-\frac{u^2}{1+r(|\tau_{\mathbf{k}\mathbf{l}}|)}\right)\\
&=H_{\alpha_1}\left(\frac{c_1^{1/\alpha_1}[0,T]^N}{(1+\rho)^{\frac{2}{\alpha_1}}}\right)H_{\alpha_2}\left(\frac{c_2^{1/\alpha_2}[0,T]^N}{(1+\rho)^{\frac{2}{\alpha_2}}}\right) \Psi(u,\rho)(1+\gamma(u))\\
&\ \ \qquad \times \sum_{(\mathbf{k},\mathbf{l})\in \mathcal{C}} \exp\left\{-u^2\left(\frac{1}{1+r(|\tau_{\mathbf{k}\mathbf{l}}|)}-\frac{1}{1+\rho}\right)\right\},
\end{split}
\end{equation}
where the global error function $\gamma(u)\to 0$, as $u\rightarrow \infty$. The uniform
convergence of \eqref{local double extremes asymptotics} in Lemma \ref{main lemma} guarantees
that the local error term $o(1)$ for each pair $(\mathbf{k},\mathbf{l})\in \mathcal{C}$ is
uniformly bounded by $\gamma(u)$.

The series in the last equality of \eqref{upper bound around maximum 2} is dealt by the following key lemma,
which gives the power of the threshold $u$  in \eqref{around maximum}.
\begin{lemma}
\label{lem_Riemann Sum}
Recall the set $\mathcal{C}$ defined in \eqref{mathcal C}. Let
\begin{equation}\label{Eq:h}
h(u):=\sum_{(\mathbf{k},\mathbf{l})\in \mathcal{C}} \exp\left\{-u^2\left(\frac{1}
{1+r(|\tau_{\mathbf{k}\mathbf{l}}|)}-\frac{1}{1+\rho}\right)\right\}.
\end{equation}
Then, under the assumptions of Theorem \ref{theorem_Jordan measurable sets}, we have
\begin{align}
\label{h(u)}
h(u)&=(2\pi)^{N/2}(-r''(0))^{-N/2}(1+\rho)^NT^{-2N}mes_N(A_1\cap A_2)\nonumber\\
&\qquad \times u^{N(\frac{2}{\alpha_1}+\frac{2}{\alpha_2}-1)}(1+o(1)),\ \ \text{ as }\ u\rightarrow \infty.
\end{align}
Moreover, if we replace $\mathcal{C}$ in (\ref{Eq:h}) by $\mathcal{C}^\circ$ defined in
\eqref{mathcal C}, then (\ref{h(u)}) still holds.
\end{lemma}

We defer the proof of Lemma \ref{lem_Riemann Sum} to Section \ref{sec_proof of lemmas}
and continue with the proof of Theorem \ref{theorem_Jordan measurable sets}.
Applying (\ref{h(u)}) to \eqref{upper bound around maximum 2}, we obtain
\begin{equation}
\label{upper bound around maximum 5}
\begin{split}
&\mathbb{P}\bigg(\bigcup_{(s,t)\in \mathcal{D}} \{X_1(s)>u, X_2(t)>u)\}\bigg)\\
&\leq (2\pi)^{\frac{N}{2}}(-r''(0))^{-\frac{N}{2}}(1+\rho)^NT^{-2N}mes_N(A_1\cap A_2)
H_{\alpha_1}\left(\frac{c_1^{1/\alpha_1}[0,T]^N}{(1+\rho)^{\frac{2}{\alpha_1}}}\right)\\
&\ \ \ \times  H_{\alpha_2}\left(\frac{c_2^{1/\alpha_2}[0,T]^N}{(1+\rho)^{\frac{2}{\alpha_2}}}\right)\,
u^{N(\frac{2}{\alpha_1}+\frac{2}{\alpha_2}-1)}\Psi(u,\rho)(1+\gamma_1(u)),
\end{split}
\end{equation}
where $\gamma_1(u)\to 0$, as $u\rightarrow \infty$. Hence,
\begin{equation}
\label{upper bound around maximum 3}
\begin{split}
&\limsup_{u\rightarrow \infty}\frac{\mathbb{P}\left(\bigcup_{(s,t)\in \mathcal{D}}
\{X_1(s)>u, X_2(t)>u)\}\right)}{u^{N(\frac{2}{\alpha_1}+\frac{2}{\alpha_2}-1)}\Psi(u,\rho)}\\
&\leq (2\pi)^{\frac{N}{2}}(-r''(0))^{-\frac{N}{2}}(1+\rho)^N mes_N(A_1\cap A_2)\\
&\qquad \times T^{-2N}H_{\alpha_1}\left(\frac{c_1^{1/\alpha_1}[0,T]^N}{(1+\rho)^{\frac{2}{\alpha_1}}}\right)H_{\alpha_2}\left(\frac{c_2^{1/\alpha_2}[0,T]^N}{(1+\rho)^{\frac{2}{\alpha_2}}}\right).
\end{split}
\end{equation}
The above inequality holds for every $T>0$. Therefore, letting $T \rightarrow \infty$, we have
\begin{align}
\label{upper bound around maximum 4}
&\limsup_{u\rightarrow \infty}\frac{\mathbb{P}\left(\bigcup_{(s,t)\in \mathcal{D}} \{X_1(s)>u, X_2(t)>u)\}\right)}{u^{N(\frac{2}{\alpha_1}+\frac{2}{\alpha_2}-1)}\Psi(u,\rho)}
\leq (2\pi)^{\frac{N}{2}}(-r''(0))^{-\frac{N}{2}} \nonumber\\
&\quad\qquad  \times c_1^{\frac{N}{\alpha_1}}c_2^{\frac{N}{\alpha_2}}(1+\rho)^{-N(\frac{2}{\alpha_1}+\frac{2}{\alpha_2}-1)}mes_N(A_1\cap A_2)H_{\alpha_1}H_{\alpha_2}.
\end{align}
On the other hand, the lower bound for LHS of \eqref{around maximum} can be derived as follows. Let
\begin{equation}
\mathcal{B}=\{(\mathbf{k},\mathbf{l},\mathbf{k}',\mathbf{l}'):\, (\mathbf{k},\mathbf{l})\neq(\mathbf{k}',\mathbf{l}'), (\mathbf{k},\mathbf{l}),(\mathbf{k}',\mathbf{l}')\in \mathcal{C}\}.
\end{equation}
By Bonferroni's inequality and symmetric property of $\mathcal{B}$, the LHS of \eqref{around maximum} is bounded below by
\begin{align}
\label{lower bound around maximum}
&\mathbb{P}\bigg(\bigcup_{(s,t)\in \mathcal{D}} \{X_1(s)>u, X_2(t)>u\}\bigg)\nonumber\\
&\geq\sum_{(\mathbf{k},\mathbf{l})\in \mathcal{C}^\circ}\mathbb{P}\bigg(\max_{s\in \Delta^{(1)}_\mathbf{k}} X_1(s)>u,\, \max_{t\in \Delta^{(2)}_\mathbf{l} } X_2(t)>u\bigg)\nonumber\\
& \qquad -\frac{1}{2}\sum_{(\mathbf{k},\mathbf{l},\mathbf{k}',\mathbf{l}')\in \mathcal{B}}\mathbb{P}\bigg(\max_{s\in \Delta^{(1)}_\mathbf{k}} X_1(s)>u,  \max_{t\in \Delta^{(2)}_\mathbf{l} } X_2(t)>u,  \\
&\qquad \qquad \qquad \max_{s\in \Delta^{(1)}_{\mathbf{k}'}} X_1(s)>u, \, \max_{t\in \Delta^{(2)}_{\mathbf{l}'} } X_2(t)>u\bigg)\nonumber\\
&\triangleq\Sigma_1-\Sigma_2. \nonumber
\end{align}
Since $\mathcal{C}^\circ$ and ${\mathcal{C}}$ are almost the same, a similar argument as in
 \eqref{upper bound around maximum 2}$\sim$ \eqref{upper bound around maximum 4} shows
 that $\Sigma_1$ is bounded below by
\begin{align}
\Sigma_1&\geq(2\pi)^{\frac{N}{2}}(-r''(0))^{-\frac{N}{2}}(1+\rho)^N mes_N(A_1\cap A_2)T^{-2N}H_{\alpha_1}\left(\frac{c_1^{1/\alpha_1}[0,T]^N}{(1+\rho)^{\frac{2}{\alpha_1}}}\right)\nonumber\\
&\qquad  \ \times H_{\alpha_2}\left(\frac{c_2^{1/\alpha_2}[0,T]^N}{(1+\rho)^{\frac{2}{\alpha_2}}}\right)\, u^{N(\frac{2}{\alpha_1}+\frac{2}{\alpha_2}-1)}\Psi(u,\rho)(1-\gamma_2(u)),
\end{align}
where $\gamma_2(u)\to 0$, as $u\rightarrow \infty$.
Hence, letting $T\rightarrow \infty$, we have
\begin{equation}
\label{lower bound around maximum 1}
\begin{split}
&\liminf_{u\rightarrow \infty}\frac{\Sigma_1}{u^{N(\frac{2}{\alpha_1}+\frac{2}{\alpha_2}-1)}\Psi(u,\rho)}\
\geq (2\pi)^{\frac{N}{2}}(-r''(0))^{-\frac{N}{2}}c_1^{\frac{N}{\alpha_1}}c_2^{\frac{N}{\alpha_2}}\\
&\qquad \qquad \times (1+\rho)^{-N(\frac{2}{\alpha_1}+\frac{2}{\alpha_2}-1)}mes_N(A_1\cap A_2)H_{\alpha_1}H_{\alpha_2}.
\end{split}
\end{equation}
Next, we consider $\Sigma_2$ in (\ref{lower bound around maximum}). To simplify the notation, we let
\begin{align}
\label{I(k,l,k',l') def}
I(\mathbf{k},\mathbf{l},\mathbf{k}',\mathbf{l}') &:=
\mathbb{P}\bigg(\max_{s\in \Delta^{(1)}_\mathbf{k}} X_1(s)>u,\, \max_{t\in \Delta^{(2)}_\mathbf{l} } X_2(t)>u,\nonumber\\
& \qquad \qquad \qquad \qquad \max_{s\in \Delta^{(1)}_{\mathbf{k}'}} X_1(s)>u, \max_{t\in \Delta^{(2)}_{\mathbf{l}'} } X_2(t)>u\bigg).\nonumber
\end{align}
For $\mathbf{m}=(m_1, \ldots,m_N)\in \mathbb{Z}^N$, let
\begin{equation}
\label{def_mathcal_H(m)}
\mathcal{H}_{\alpha,c}(\mathbf{m})\triangleq H_\alpha\left(\frac{c^{1/\alpha}[0,T]^N}{(1+\rho)^{\frac{2}{\alpha}}},
\frac{c^{1/\alpha}[\mathbf{m}T,(\mathbf{m}+1)T]}{(1+\rho)^{\frac{2}{\alpha}}}\right).
\end{equation}
Rewriting $\Sigma_2$ and applying Lemma \ref{double double local extremes lemma}, we obtain
\begin{align}
\label{lower bound around maximum 2}
\Sigma_2&=\frac{1}{2}\sum_{(\mathbf{k},\mathbf{l})\in \mathcal{C}}\bigg(\sum_{\substack{(\mathbf{k}',\mathbf{l}')\in \mathcal{C}\\\mathbf{k}'=\mathbf{k},\mathbf{l}'\neq \mathbf{l}}}
+\sum_{\substack{(\mathbf{k}',\mathbf{l}')\in \mathcal{C}\\ \mathbf{k}'\neq \mathbf{k},\mathbf{l}'=\mathbf{l}}}
+\sum_{\substack{(\mathbf{k}',\mathbf{l}')\in \mathcal{C}\\\mathbf{k}'\neq \mathbf{k},\mathbf{l}'\neq \mathbf{l}}}\bigg)I(\mathbf{k},\mathbf{l},\mathbf{k}',\mathbf{l}') \nonumber\\
=& \frac{(1+\rho)^2(1+\gamma_3(u))}{4\pi \sqrt{1-\rho^2}\,u^2}\sum_{(\mathbf{k},\mathbf{l})\in \mathcal{C}} e^{-\frac{u^2}{1+r(|\tau_{\mathbf{k}\mathbf{l}}|)}}\, \bigg(\mathcal{H}_{\alpha_1,c_1}(\mathbf{0})\sum_{\substack{(\mathbf{k}',\mathbf{l}')\in \mathcal{C}\\\mathbf{k}'=\mathbf{k},\mathbf{l}'\neq \mathbf{l}}}\mathcal{H}_{\alpha_2,c_2}(\mathbf{l}'-\mathbf{l})\nonumber\\
&+\mathcal{H}_{\alpha_2,c_2}(\mathbf{0})\sum_{\substack{(\mathbf{k}',\mathbf{l}')\in \mathcal{C}\\\mathbf{k}'\neq \mathbf{k},\mathbf{l}'=\mathbf{l}}}\mathcal{H}_{\alpha_1,c_1}(\mathbf{k}'-\mathbf{k})+
\sum_{\substack{(\mathbf{k}',\mathbf{l}')\in \mathcal{C}\\\mathbf{k}'\neq \mathbf{k},\mathbf{l}'\neq \mathbf{l}}}\mathcal{H}_{\alpha_1,c_1}(\mathbf{k}'-\mathbf{k})\mathcal{H}_{\alpha_2,c_2}(\mathbf{l}'-\mathbf{l})\bigg)\nonumber\\
\leq &\frac{(1+\rho)^2(1+\gamma_3(u))}{4\pi \sqrt{1-\rho^2}u^2}\sum_{(\mathbf{k},\mathbf{l})\in \mathcal{C}}e^{-\frac{u^2}{1+r(|\tau_{\mathbf{k}\mathbf{l}}|)}}
\bigg(\mathcal{H}_{\alpha_1,c_1}(\mathbf{0})\sum_{\mathbf{n}\neq \mathbf{0}}\mathcal{H}_{\alpha_2,c_2}(\mathbf{n})\nonumber\\
&+\mathcal{H}_{\alpha_2,c_2}(\mathbf{0})\sum_{\mathbf{m}\neq \mathbf{0}}\mathcal{H}_{\alpha_1,c_1}(\mathbf{m})+\sum_{\mathbf{m}\neq \mathbf{0},\mathbf{n}\neq \mathbf{0}}\mathcal{H}_{\alpha_1,c_1}(\mathbf{m})\mathcal{H}_{\alpha_2,c_2}(\mathbf{n})\bigg),
\end{align}
where $\gamma_3(u)\to 0$, as $u\rightarrow \infty$. According to the uniform convergence of 
\eqref{double double local extremes asymptotics}, the local error term $o(1)$ for each pair 
$(\mathbf{k'},\mathbf{l'})\in \mathcal{C}$ is bounded above by $\gamma_3(u)$ . To estimate 
$\mathcal{H}_{\alpha,c}(\cdot)$, we make use of the following lemma, whose proof is again postponed to Section \ref{sec_proof of lemmas}.
\begin{lemma}
\label{lem_mathcal_H(m)}
Recall $\mathcal{H}_{\alpha,c}(\cdot)$ defined in \eqref{def_mathcal_H(m)}. Let $i_0=
{\rm argmax}_{1\leq i \leq N}|m_i|$. Then  there exist positive constants
$C_1$ and $T_0$ such that for all $T \ge T_0$,
\begin{align}
\label{mathcal_H(0) bounds}
&\mathcal{H}_{\alpha,c}(\mathbf{0})\leq  C_1 T^N;\\
\label{mathcal_H(1) bounds}
&\mathcal{H}_{\alpha,c}(\mathbf{m})\leq C_1 T^{N-\frac{1}{2}},\ \text{when}\ |m_{i_0}|=1;\\
\label{mathcal_H(m) bounds}
&\mathcal{H}_{\alpha,c}(\mathbf{m})\leq C_1 T^{2N}e^{-\frac{c}{8(1+\rho)^2}(|m_{i_0}|-1)^\alpha T^\alpha}, \ \text{when}\ |m_{i_0}|\geq 2.
\end{align}
Consequently,
\begin{equation}
\label{mathcal_H(m) summation bounds}
\sum_{\mathbf{m}\in \mathbb{Z}^N\setminus\{\mathbf{0}\}}\mathcal{H}_{\alpha,c}(\mathbf{m})\leq C_1 T^{N-\frac{1}{2}}.
\end{equation}
\end{lemma}
Applying Lemmas \ref{lem_Riemann Sum} and \ref{lem_mathcal_H(m)} to the RHS of \eqref{lower bound around maximum 2},
we obtain
\begin{align}
\Sigma_2&\leq \frac{C_0(1+\rho)^2(1+\gamma_3(u))}{4\pi \sqrt{1-\rho^2}\, u^2}T^{2N-\frac{1}{2}}
\sum_{(\mathbf{k},\mathbf{l})\in \mathcal{C}}\exp\left(-\frac{u^2}{1+r(|\tau_{\mathbf{k}\mathbf{l}}|)}\right)\nonumber\\
&\leq C_0 (2\pi)^{\frac{N}{2}}(-r''(0))^{-\frac{N}{2}}(1+\rho)^N mes_N(A_1\cap A_2)T^{-\frac{1}{2}} \nonumber\\
&\qquad \qquad \times  u^{N(\frac{2}{\alpha_1}+\frac{2}{\alpha_2}-1)}\Psi(u,\rho)(1+\gamma_4(u)),
\end{align}
where $\gamma_4(u)\to 0$, as $u\rightarrow \infty$. By letting $u\rightarrow \infty$
and $T\rightarrow\infty$ successively, we have
\begin{align}
\label{lower bound around maximum 22}
&\limsup_{u\rightarrow \infty}\frac{\Sigma_2}{u^{N(\frac{2}{\alpha_1}+\frac{2}{\alpha_2}-1)}\Psi(u,\rho)}=0.
\end{align}
By combining \eqref{lower bound around maximum}, \eqref{lower bound around maximum 1} and 
\eqref{lower bound around maximum 22}, we have
\begin{align}
\label{lower bound around maximum 3}
&\liminf_{u\rightarrow \infty}\frac{\mathbb{P}\left(\bigcup_{(s,t)\in \mathcal{D}} 
\{X_1(s)>u,X_2(t)>u)\}\right)}{u^{N(\frac{2}{\alpha_1}+\frac{2}{\alpha_2}-1)}\Psi(u,\rho)}\nonumber\\
&\geq \liminf_{u\rightarrow \infty}\frac{\Sigma_1}{u^{N(\frac{2}{\alpha_1}+\frac{2}{\alpha_2}-1)}\Psi(u,\rho)}-\limsup_{u\rightarrow \infty}\frac{\Sigma_2}{u^{N(\frac{2}{\alpha_1}+\frac{2}{\alpha_2}-1)}\Psi(u,\rho)}\\
&\geq (2\pi)^{\frac{N}{2}}(-r''(0))^{-\frac{N}{2}}c_1^{\frac{N}{\alpha_1}}c_2^{\frac{N}{\alpha_2}}(1+\rho)^{-N(\frac{2}{\alpha_1}+\frac{2}{\alpha_2}-1)}mes_N(A_1\cap A_2)H_{\alpha_1}H_{\alpha_2}. \nonumber
\end{align}
It is now clear that \eqref{around maximum} follows from  \eqref{upper bound around maximum 4} and \eqref{lower bound around maximum 3}.

Now we prove \eqref{off maximum}.  Define
\begin{equation}
\label{Y(s,t)}
Y(s,t):=X_1(s)+X_2(t),\ \text{for}\ (s,t)\in \Pi\setminus \mathcal{D}.
\end{equation}
For $x=(s_1,t_1), y=(s_2,t_2) \in \Pi\setminus \mathcal{D} $, let
$|x-y|=\sqrt{|s_1-s_2|^2+|t_1-t_2|^2}$. Then we can verify that
\begin{equation}
\label{global holder}
\mathbb{E}|Y(x)-Y(y)|^2\leq C_0 |x-y|^{\min(\alpha_1,\alpha_2)}, \ \forall x,y\in \Pi\setminus \mathcal{D}.
\end{equation}
By applying Theorem $8.1$ in \cite{Piterbarg_1996}, we obtain that the numerator of \eqref{off maximum} is bounded above by
\begin{eqnarray}
\label{off maximum 1}
&&\mathbb{P}\bigg(\bigcup_{(s,t)\in\Pi\setminus \mathcal{D}} \{X_1(s)>u,X_2(t)>u)\}\bigg)\leq \mathbb{P}\left(\max_{(s,t)\in\Pi\setminus \mathcal{D}}Y(s,t)>2u\right)\nonumber\\
&&\leq C_0u^{-1+\frac{2N}{\min(\alpha_1,\alpha_2)}}\exp\left(-\frac{u^2}{1+\max_{(s,t)\in\Pi\setminus \mathcal{D}}r(|t-s|)}\right).
\end{eqnarray}
Since $r(|t-s|)=\rho+\frac{1}{2}r''(0)|t-s|^2(1+o(1))$ and $r(\cdot)$ attains maximum only at zero, we have
\begin{eqnarray}
\label{off maximum 2}
\max_{(s,t)\in\Pi\setminus \mathcal{D}}r(|t-s|) \leq \rho-\frac{1}{3}(-r''(0))\delta^2(u)
\end{eqnarray}
for $ u$  large enough. So \eqref{off maximum 1} is at most
\begin{equation}
\label{off maximum 3}
\begin{split}
&C_0u^{-1+\frac{2N}{\min(\alpha_1,\alpha_2)}}\exp\left(-\frac{u^2}{1+\rho-\frac{1}{3}(-r''(0))\delta^2(u)}\right) \\
&\leq  C_0u^{-1+\frac{2N}{\min(\alpha_1,\alpha_2)}}\exp\left(-\frac{u^2}{1+\rho}\right)\exp\left(-\frac{\frac{1}{3}(-r''(0))\delta^2(u)u^2}{(1+\rho)^2}\right) \\
&=\frac{2\pi\sqrt{1-\rho^2}C_0}{(1+\rho)^2}u^{1+\frac{2N}{\min(\alpha_1,\alpha_2)}-\frac{-r''(0)}{3(1+\rho)^2}C^2}\Psi(u,\rho),
\end{split}
\end{equation}
where the inequality holds since $\frac{1}{x-y}\geq \frac{1}{x}+\frac{y}{x^2}, \forall x>y$.
Compare \eqref{off maximum 3} with \eqref{around maximum}, it is easy to see \eqref{off maximum} holds if and only if
\begin{equation}
1+\frac{2N}{\min(\alpha_1,\alpha_2)}-\frac{-r''(0)}{3(1+\rho)^2}C^2<N\Big(\frac{2}{\alpha_1}+\frac{2}{\alpha_2}-1\Big)
\end{equation}
Hence, by choosing the constant $C$ satisfying
\begin{equation}
C>\left[\frac{3(1+\rho)^2}{-r''(0)}\left(N\Big(\frac{2}{\min(\alpha_1,\alpha_2)}+1-\frac{2}{\alpha_1}-\frac{2}{\alpha_2}\Big)+1\right)_{+}\right]^\frac{1}{2},
\end{equation}
we  conclude \eqref{off maximum}.
\end{proof}

\begin{proof}[Proof of Theorem \ref{theorem: Jordan measurable sets with mes zero}]
From the proof of Theorem \ref{theorem_Jordan measurable sets}, we see that
the exponential decaying rate of the excursion probability is only determined
by the region where the maximum cross correlation is attained. In the case
of $mes_N(A_1\cap A_2)=0$ but $
A_1\cap A_2 \neq \emptyset$, the exponential part, $e^{-\frac{u^2}{1+\rho}}$,
remains the same.  Yet, the dimension reduction of $A_1\cap A_2$ does affect
the polynomial power of the excursion probability, which is determined by the quantity $$h(u)=\sum_{(\mathbf{k},\mathbf{l})\in \mathcal{C}} \exp\left\{-u^2\left(\frac{1}{1+r(|\tau_{\mathbf{k}\mathbf{l}}|)}-\frac{1}{1+\rho}\right)\right\}$$
in Lemma \ref{lem_Riemann Sum}. Under the assumptions of Theorem
\ref{theorem: Jordan measurable sets with mes zero}, the
set $\mathcal{C}$ and the behavior of $h(u)$ change. We will make use of
the following lemma which plays the role of Lemma \ref{lem_Riemann Sum}.

\begin{lemma}
\label{lem_Riemann Sum-2}
Under the assumptions of Theorem \ref{theorem: Jordan measurable sets with mes zero}, we have
\begin{align}
\label{h(u)_mes zero}
h(u)&=(2\pi)^{M/2}(-r''(0))^{M/2-N}(1+\rho)^{2N-M}T^{-2N}mes_M(A_{1,M}\cap A_{2,M})\nonumber\\
&\times u^{M+N\big(\frac{2}{\alpha_1}+\frac{2}{\alpha_2}-2\big)}(1+o(1)),\ \text{as}\ u\rightarrow \infty.
\end{align}
Moreover, if we replace $\mathcal{C}$ with $\mathcal{C}^\circ$ defined in \eqref{mathcal C}, then
the above statement still holds.
\end{lemma}

The rest of the proof of Theorem \ref{theorem: Jordan measurable sets with mes zero} is the same as
 that of Theorem \ref{theorem_Jordan measurable sets} and it is omitted here.
\end{proof}

\section{Proof of Lemmas}
\label{sec_proof of lemmas}

For proving Lemma \ref{main lemma}, we will make use of the following
\begin{lemma}
\label{uniformly conditional convergence}
Let $s_u$ and $t_u$ be two $\mathbb{R}^N$-valued functions of $u$ and
let $\tau_u:=t_u-s_u$.
For any  compact rectangles  $\mathbb{S}$ and
$\mathbb{T}$ in $\mathbb{R}^N$, define
\begin{align}
\label{tangent field}
\xi_u(s)&:=u(X_1(s_u+u^{-2/\alpha_1}s)-u)+x,\quad \forall \, s\in \mathbb{S},\nonumber\\
\eta_u(t)&:=u(X_2(t_u+u^{-2/\alpha_2}t)-u)+y,\quad  \forall \, t\in \mathbb{T}
\end{align}
and for any $t \in \mathbb{R}^N,$ let
\begin{equation}
\label{bivariate FBM}
\xi(t):=\sqrt{c_1}\chi_1(t)-\frac{c_1|t|^{\alpha_1}}{1+\rho},\ \ \
 \eta(t):=\sqrt{c_2}\chi_2(t)-\frac{c_2|t|^{\alpha_2}}{1+\rho},
\end{equation}
where  $\chi_1(t),\chi_2(t)$ are two independent
fractional Brownian motions with indices $\alpha_1/2$ and $\alpha_2/2$, respectively.
Then, the finite dimensional distributions (abbr. f.d.d.) of $(\xi_u(\cdot),\eta_u(\cdot))$,
given $X_1(s_u)=u-\frac{x}{u},\, X_2(t_u)=u-\frac{y}{u}$, converge
uniformly to the f.d.d. of $(\xi(\cdot),\eta(\cdot))$ for all $s_u$ and $t_u$ that satisfy
$|\tau_u|\leq C_0\sqrt{\log u}/u$.  Furthermore, as $u \rightarrow \infty$,
\begin{align}
\label{cond maximum unif convergence}
&\mathbb{P}\bigg(\max_{s\in \mathbb{S}}\xi_u(s)>x,\max_{t\in \mathbb{T}}\eta_u(t)>y\ \Big |\ X_1(s_u)
=u-\frac{x}{u},X_2(t_u)=u-\frac{y}{u}\bigg)\nonumber\\
&\rightarrow \mathbb{P}\bigg(\max_{s\in \mathbb{S}}\xi(s)>x,\max_{t\in \mathbb{T}}\eta(t)>y\bigg),
\end{align}
where the convergence is uniform for all $s_u$ and $t_u$ that satisfy $|\tau_u|\leq C_0\sqrt{\log u}/u$.
\end{lemma}

\begin{proof}
First, we prove the uniform convergence of finite dimensional distributions.
Given $X_1(s_u)=u-\frac{x}{u},\, X_2(t_u)=u-\frac{y}{u}$,
the distribution of the bivariate random field $(\xi_u(\cdot),\eta_u(\cdot))$
is still Gaussian. Thanks to the following
lemma (whose proof will be given at the end of this section), it suffices to
prove the uniform convergence of conditional mean and conditional variance.

\begin{lemma}
\label{lem_uniform convergence of f.d.d.}
Let $X(u,\tau_u)=(X_1(u,\tau_u),\ldots,X_n(u,\tau_u))^T$ be a Gaussian random
vector with mean $\mu(u,\tau_u)=(\mu_1(u,\tau_u),...,\mu_n(u,\tau_u)^T$ and
covariance matrix $\Sigma(u,\tau_u)$ with entries $\sigma_{ij}(u,\tau_u)
={\rm Cov}(X_i(u,$ $\tau_u),X_j(u,\tau_u)),\ i,j=1,2,\dots,n$. Similarly, let
$X=(X_1,\ldots,X_n)^T$ be a Gaussian random vector with mean $\mu=(\mu_1,...,\mu_n)$
and covariance matrix $\Sigma=(\sigma_{ij})_{i,j=1,...,n}$. Assume that $\Sigma$ is non-singular. Let $F_u(\cdot)$ and
$F(\cdot)$ be the distribution functions of $X(u,\tau_u)$ and $X$ respectively.
If
\begin{align}
\label{mean_var_unif_conver}
&\lim_{u\rightarrow \infty} \max_{\tau_u}|\mu_j(u,\tau_u)-\mu_j|=0,\nonumber\\
&\lim_{u\rightarrow \infty} \max_{\tau_u}|\sigma_{ij}(u,\tau_u)-\sigma_{ij}|=0,\ \  
i,j=1,2,\dots, n,
\end{align}
then for any $x\in \mathbb{R}^N$,
\begin{align}
\lim_{u\rightarrow \infty} \max_{\tau_u}|F_u(x)-F(x)|=0.
\end{align}
\end{lemma}

We continue with the proof of Lemma \ref{uniformly conditional convergence} and postpone the
proof of Lemma \ref{lem_uniform convergence of f.d.d.} to the end of this section.  
Recall that, for two random vectors $X,Y\in \mathbb{R}^m$, their covariance is defined as
${\rm Cov}(X,Y):=\mathbb{E}[(X-\mathbb{E}X)(Y-\mathbb{E}Y)^T]$ and the variance matrix
of $X$ is defined as ${\rm Var}(X):={\rm Cov}(X,X)$. The conditional mean of $(\xi_u(t),\eta_u(t))^T$
given $X_1(s_u)=u-\frac{x}{u},\, X_2(t_u)=u-\frac{y}{u}$, is
\begin{align}
\label{conditional mean}
&\mathbb{E}\left(
\begin{array}{l}
\xi_u(t)\\
\eta_u(t)
\end{array}
\left |
\begin{array}{l}
X_1(s_u)=u-\frac{x}{u}\\
X_2(t_u)=u-\frac{y}{u}
\end{array}
\right.
\right)  = \mathbb{E}\bigg(
\begin{array}{l}
\xi_u(t)\\
\eta_u(t)
\end{array}
\bigg) \nonumber\\
& \qquad +{\rm Cov}\left(\left(
\begin{array}{l}
\xi_u(t)\\
\eta_u(t)
\end{array}
\right),
\left(
\begin{array}{ll}
X_1(s_u)\\
X_2(t_u)
\end{array}
\right)
\right) \left({\rm Var}\left(
\begin{array}{ll}
X_1(s_u)\\
X_2(t_u)
\end{array}
\right)
\right)^{-1}\left(
\begin{array}{l}
u-\frac{x}{u}\\
u-\frac{y}{u}
\end{array}
\right)\nonumber\\
&=\left(
\begin{array}{l}
-u^2+x\\
-u^2+y
\end{array}
\right)+\frac{u}{1-r^2(|\tau_u|)}\left(
\begin{array}{ll}
r_{11}(s_u+u^{-2/\alpha_1}t,s_u)&r(|\tau_u-u^{-2/\alpha_1}t|)\\
r(|\tau_u+u^{-2/\alpha_2}t|)& r_{22}(t_u+u^{-2/\alpha_2}t,t_u)
\end{array}
\right)\nonumber\\
&\qquad \times
\left(
\begin{array}{cc}
1&-r(|\tau_u|)\\
-r(|\tau_u|)&1
\end{array}
\right)\left(
\begin{array}{l}
u-\frac{x}{u}\\
u-\frac{y}{u}
\end{array}
\right)\nonumber\\
& \triangleq \left(
\begin{array}{l}
a_1(u)\\
a_2(u)
\end{array}
\right),
\end{align}
where
\begin{align}
a_1(u)=&
-\frac{u^2(1-r_{11}(s_u+u^{-2/\alpha_1}t,s_u))-u^2(r(|\tau_u-u^{-2/\alpha_1}t|)-r(|\tau_u|))}{1+r(|\tau_u|)}\nonumber\\
&+\frac{(x-yr(|\tau_u|))(1-r_{11}(s_u+u^{-2/\alpha_1}t,s_u))}{1-r^2(|\tau_u|)}\nonumber\\
&+\frac{(y-xr(|\tau_u|))(r(|\tau_u|)-r(|\tau_u-u^{-2/\alpha_1}t|))}{1-r^2(|\tau_u|)}
\end{align}
and
\begin{align}
a_2(u)=&-\frac{u^2(1-r_{22}(t_u+u^{-2/\alpha_2}t,t_u))-u^2(r(|\tau_u+u^{-2/\alpha_2}t|)-r(|\tau_u|))}{1+r(|\tau_u|)}\nonumber\\
&+\frac{(y-xr(|\tau_u|))(1-r_{22}(t_u+u^{-2/\alpha_1}t,t_u))}{1-r^2(|\tau_u|)}\nonumber\\
&+\frac{(x-yr(|\tau_u|))(r(|\tau_u|)-r(|\tau_u+u^{-2/\alpha_2}t|))}{1-r^2(|\tau_u|)}.
\end{align}
Applying the mean value theorem twice, we see that for $u$ large enough,
\begin{eqnarray}
\label{cross corr diff}
&&|r(|\tau_u+u^{-2/\alpha}t|)-r(|\tau_u|)|\leq |u^{-2/\alpha}t|\cdot \max_{\substack{s\  \text{is between}\\
|\tau_u|\ \text{and}\ |\tau_u+u^{-2/\alpha}t|}}|r'(s)|\nonumber\\
&&\leq |u^{-2/\alpha}t|\cdot \max_{|s|\leq 2C_0\sqrt{\log u}/u}|r'(s)| \nonumber\\
&&\leq|u^{-2/\alpha}t|\cdot \max_{|s|\leq 2C_0\sqrt{\log u}/u}\left(|s|\cdot\max_{ |t|\leq |s|}|r''(t)|\right)\nonumber\\
&&\leq2 C_0|t|\sqrt{\log u}\cdot u^{-1-2/\alpha}\cdot \max_{|t| \leq 2 C_0\sqrt{\log u}/u}|r''(t)|\nonumber\\
&&\leq 4 C_0|r''(0)||t|\sqrt{\log u}\cdot u^{-1-2/\alpha},
\end{eqnarray}
where the second inequality holds because of $u^{-2/\alpha}=o(\sqrt{\log u}/u)$, as $u \rightarrow \infty$
and the last inequality holds since $r''(\cdot)$ is continuous in a neighborhood of zero. Thus
\eqref{cross corr diff} implies that, as $u \rightarrow \infty$,
\begin{equation}
\label{cross corr diff 2}
u^2|r(|\tau_u+u^{-2/\alpha}t|)-r(|\tau_u|)|\leq 4 C_0|r''(0)||t|\sqrt{\log u}\cdot u^{1-2/\alpha}\rightarrow 0,
\end{equation}
where the convergence is uniform for all $s_u$ and $t_u$ that satisfy $|\tau_u|\leq C_0\sqrt{\log u}/u$.
We also notice that for $ i=1,2$ and all $s\in \mathbb{R}^N$,
\begin{equation}
\label{auto corr diff}
1-r_{ii}(s+u^{-2/\alpha}t,s)=c_iu^{-2}|t|^{\alpha_i}+o(u^{-2}), \, \text{as}\ u\rightarrow \infty.
\end{equation}
By \eqref{conditional mean}, \eqref{cross corr diff 2}, and \eqref{auto corr diff}, we conclude that, as
$u \rightarrow \infty,$
\begin{align}
\label{conditional mean 2}
&\mathbb{E}\left(
\begin{array}{l}
\xi_u(t)\\
\eta_u(t)
\end{array}
\left |
\begin{array}{l}
X_1(s_u)=u-\frac{x}{u}\\
X_2(t_u)=u-\frac{y}{u}
\end{array}
\right.
\right)\rightarrow\left(
\begin{array}{l}
-\frac{c_1|t|^{\alpha_1}}{1+\rho}\\
-\frac{c_2|t|^{\alpha_2}}{1+\rho}
\end{array}
\right),
\end{align}
where the convergence is uniform w.r.t. $s_u$ and $t_u$  satisfying $|\tau_u|\leq C_0\sqrt{\log u}/u$.

Next, we consider the conditional covariance matrix of $(\xi_u(t)-\xi_u(s),\eta_u(t)-\eta_u(s))^T$.
\begin{align}
\label{conditional variance}
&{\rm Var}\left(\left(
\begin{array}{l}
\xi_u(t)-\xi_u(s)\\
\eta_u(t)-\eta_u(s)
\end{array}
\right)\left |
\begin{array}{l}
X_1(s_u)\\
X_2(t_u)
\end{array}
 \right.
\right)\nonumber\\
&={\rm Var}\left(
\begin{array}{l}
\xi_u(t)-\xi_u(s)\\
\eta_u(t)-\eta_u(s)
\end{array}
\right)
- {\rm Cov}\left(\left(
\begin{array}{l}
\xi_u(t)-\xi_u(s)\\
\eta_u(t)-\eta_u(s)
\end{array}
\right),
\left(
\begin{array}{ll}
X_1(s_u)\\
X_2(t_u)
\end{array}
\right)
\right)\nonumber\\
&\qquad \times{\rm Var}\left(
\begin{array}{l}
X_1(s_u)\\
X_2(t_u)
\end{array}
\right)^{-1} {\rm Cov}\left(\left(
\begin{array}{l}
\xi_u(t)-\xi_u(s)\\
\eta_u(t)-\eta_u(s)
\end{array}
\right),
\left(
\begin{array}{ll}
X_1(s_u)\\
X_2(t_u)
\end{array}
\right)
\right)^T.
\end{align}
Let $h_u(t,s):=r(|\tau_u+u^{-2/\alpha_2}t-u^{-2/\alpha_1}s|)$. Applying \eqref{cross corr diff 2} and
\eqref{auto corr diff}, we obtain
\begin{align}
\label{conditional variance 1}
&{\rm Var}\left(
\begin{array}{l}
\xi_u(t)-\xi_u(s)\\
\eta_u(t)-\eta_u(s)
\end{array}
\right)\nonumber\\
&=\left(
\begin{array}{ll}
2u^2(1-r_{11}(s_u+  &u^2(h_u(t,t)-h_u(s,t)\\
\quad u^{-2/\alpha_1}s,\, s_u+u^{-2/\alpha_1}t))&\quad -h_u(t,s)+h_u(s,s))\\
\vspace{-2mm}\\
u^2(h_u(t,t)-h_u(s,t)&2u^2(1-r_{22}(t_u+\\
\quad-h_u(t,s)+h_u(s,s))&\quad u^{-2/\alpha_2}s,\, t_u+u^{-2/\alpha_2}t))
\end{array}
\right)\nonumber\\
&=\left(
\begin{array}{cc}
2c_1|t-s|^{\alpha_1}(1+o(1))&o(1)\\
o(1)&2c_2|t-s|^{\alpha_2}(1+o(1))
\end{array}
\right),
\end{align}
where $o(1)$ converges to zero uniformly w.r.t. $\tau_u$ satisfying 
$|\tau_u|\leq C_0\sqrt{\log u}/u$, as $u \rightarrow \infty$.
Also, we have
\begin{align}
\label{conditional variance 2}
&{\rm Cov}\left[\left(
\begin{array}{l}
\xi_u(t)-\xi_u(s)\\
\eta_u(t)-\eta_u(s)
\end{array}\right),
\left(
\begin{array}{l}
X_1(s_u)\nonumber\\
X_2(t_u)
\end{array}
\right)\right]\nonumber\\
&=\left(
\begin{array}{ll}
u(r_{11}(s_u+u^{-2/\alpha_1}t,s_u)&u(r(|\tau_u-u^{-2/\alpha_1}t|)\\
\quad -r_{11}(s_u+u^{-2/\alpha_1}s,s_u))&\quad -r(|\tau_u-u^{-2/\alpha_1}s|))\\
\vspace{-2mm}\\
u(r(|\tau_u+u^{-2/\alpha_2}t|)&u(r_{22}(t_u+u^{-2/\alpha_2}t,t_u)\\
\quad-r(|\tau_u+u^{-2/\alpha_2}s|))&\quad -r_{22}(t_u+u^{-2/\alpha_2}s,t_u))
\end{array}
\right)\nonumber\\
&=\left(
\begin{array}{ll}
o(1)&o(1)\\
o(1)&o(1)
\end{array}
\right),
\end{align}
$\text{as}\ u\rightarrow \infty,$ and
\begin{align}
\label{conditional variance 3}
&{\rm Var}\left(
\begin{array}{l}
X_1(s_u)\\
X_2(t_u)
\end{array}
\right)^{-1}=\frac{1}{1-r^2(|\tau_u|)}
\left(
\begin{array}{ll}
1&-r(|\tau_u|)\\
-r(|\tau_u|)&1
\end{array}
\right).
\end{align}
By \eqref{conditional variance} -- \eqref{conditional variance 3}, we conclude that as $u\rightarrow \infty$,
\begin{align}
\label{conditional variance 4}
{\rm Var}\left(\left(
\begin{array}{l}
\xi_u(t)-\xi_u(s)\\
\eta_u(t)-\eta_u(s)
\end{array}
\right)\bigg |
\begin{array}{l}
X_1(s_u)\\
X_2(t_u)
\end{array}
\right)\rightarrow\left(
\begin{array}{cc}
2c_1|t-s|^{\alpha_1}&0\\
0&2c_2|t-s|^{\alpha_2}
\end{array}
\right),
\end{align}
where the convergence is uniform w.r.t.  $\tau_u$ satisfying $|\tau_u|\leq C_0\sqrt{\log u}/u$. Hence,
the uniform convergence of f.d.d. in Lemma \ref{uniformly conditional convergence} follows from
 \eqref{conditional mean 2}, \eqref{conditional variance 4} and Lemma \ref{lem_uniform convergence of f.d.d.}.

Now we prove the second part of Lemma  \ref{uniformly conditional convergence}. The continuous
mapping theorem (see, e.g., \cite{Billingsley_1968}, p. $30$)
can be used to prove \eqref{cond maximum unif convergence} holds when $s_u$ and $t_u$ are fixed.
Since we need to prove uniform convergence w.r.t. $s_u$ and $t_u$, we use a discretization method
instead. Let
\begin{align}
\label{f(u,x,y)}
&f(u,x,y):=\mathbb{P}\bigg(\max_{s\in \mathbb{S}}\xi_u(s)>x,\, \max_{t\in \mathbb{T}}\eta_u(t)>y\ \Big |\nonumber\\
&\qquad \qquad \qquad \qquad  X_1(s_u)=u-\frac{x}{u}, X_2(t_u)=u-\frac{y}{u}\bigg)
\end{align}
and
\begin{align}
\label{f(x,y)}
&f(x,y)
:=\mathbb{P}\Big(\max_{s\in \mathbb{S}}\xi(s)>x,\, \max_{t\in\mathbb{T}}\eta(t)>y\Big).
\end{align}
Without loss of generality, suppose that $\mathbb{S}=[a,b]^N$ and $\mathbb{T}=[c,d]^N$,
where $a<b,c<d$. For any $\delta\in (0,1),$ let $m=\left\lfloor \frac{b-a}{\delta}\right\rfloor$, $n=\left\lfloor \frac{d-c}{\delta}\right\rfloor$
and let
\begin{align*}
\mathcal{S}_m&:= \big\{s_{\mathbf{k}}\ |\ s_{\mathbf{k}}=(x_{k_1},\,..., \,x_{k_N}), \ \mathbf{k}=(k_1,...,k_N)\in \{0,1,...,m+1\}^N\big\},\\
\mathcal{T}_n&:= \big\{t_\mathbf{l}\ |\ t_\mathbf{l}=(y_{l_1},\,...,\,y_{l_N}), \ \mathbf{l}=(l_1,...,l_N)\in \{0,1,...,n+1\}^N \big\},
\end{align*}
where $x_i, y_i$ are defined as
\begin{align}
a&=x_0<x_1<\cdots<x_m\leq x_{m+1}=b,\ x_i=a+i\delta, \ i=0,1,\ldots,m,\nonumber\\
c&=y_0<y_1<\cdots<y_n\leq y_{n+1}=d, \ \ y_i=c+i\delta, \ i=0,1,\ldots,n.
\end{align}
Then $[a,b]^N\times[c,d]^N$ can be divided into $\delta$-cubes with  vertices in $\mathcal{S}_m\times \mathcal{T}_n$.

The function $f(u,x,y)$ in \eqref{f(u,x,y)}  is bounded below by
\begin{align}
\label{lower bound of cmuc}
&f_{m,n}(u,x,y):=\mathbb{P}\bigg(\max_{s\in \mathcal{S}_m}\xi_u(s)>x,\, \max_{t\in \mathcal{T}_n}\eta_u(t)>y\ \Big | \nonumber\\
&\qquad\qquad\qquad\qquad \qquad \ X_1(s_u)=u-\frac{x}{u},X_2(t_u)=u-\frac{y}{u}\bigg)
\end{align}
and is bounded above by $g_{m,n}(u,x,y)$ which is defined as
\begin{align}
\label{upper bound of cmuc}
&\mathbb{P}\bigg(\max_{s\in \mathcal{S}_m}\xi_u(s)>x-\epsilon,\, \max_{t\in \mathcal{T}_n}\eta_u(t)>y-\epsilon\,
\Big |\, X_1(s_u)=u-\frac{x}{u},X_2(t_u)=u-\frac{y}{u}\bigg)\nonumber\\
&+\mathbb{P}\bigg(\max_{s\in \mathbb{S}}\xi_u(s)> x,\, \max_{s\in \mathcal{S}_m}\xi_u(s)\leq x-\epsilon\,
\Big|\,   X_1(s_u)=u-\frac{x}{u},X_2(t_u)=u-\frac{y}{u}\bigg)\nonumber\\
&+\mathbb{P}\bigg(\max_{t\in \mathbb{T}}\eta_u(t)>y,\, \max_{t\in \mathcal{T}_n}\eta_u(t)\leq y-\epsilon\,
\Big|\,   X_1(s_u)=u-\frac{x}{u},X_2(t_u)=u-\frac{y}{u}\bigg)\nonumber\\
& \triangleq  f_{m,n}(u,x-\epsilon,y-\epsilon)+s_{m,n}(u,x,y)+t_{m,n}(u,x,y),
\end{align}
where $\epsilon>0$ is any small constant. Let
\begin{align}
\label{lower bound of cmuc 1}
&f_{m,n}(x,y)
:=\mathbb{P}\bigg(\max_{s\in \mathcal{S}_m}\xi(s)>x,\, \max_{t\in \mathcal{T}_n}\eta(t)>y\bigg).
\end{align}
Since the finite dimensional distributions of $(\xi_u(\cdot), \eta_u(\cdot))$ converge uniformly to those
of $(\xi(\cdot),\eta(\cdot))$, we have
\begin{equation}
\label{lower bound of cmuc 3}
\lim_{u\rightarrow \infty} \max_{|\tau_u|\leq C_0\sqrt{\log u}/u}|f_{m,n}(u,x,y)-f_{m,n}(x,y)|=0.
\end{equation}
The continuity of the trajectory of $(\xi(\cdot),\eta(\cdot))$  yields
\begin{equation}
\label{lower bound of cmuc 4}
\lim_{\substack {m\rightarrow \infty \\ n\rightarrow \infty }}f_{m,n}(x,y)=f(x,y).
\end{equation}
By \eqref{lower bound of cmuc 3} and \eqref{lower bound of cmuc 4}, we conclude
\begin{equation}
\label{lower bound of cmuc 5}
\lim_{\substack {m\rightarrow \infty \\ n\rightarrow \infty }}\lim_{u\rightarrow \infty} \max_{|\tau_u|\leq
C_0\sqrt{\log u}/u}|f_{m,n}(u,x,y)-f(x,y)|=0.
\end{equation}
Let us consider the conditional probability $s_{m,n}(u,x,y)$ in \eqref{upper bound of cmuc}.
\begin{align}
\label{upper bound of cmuc 1}
&s_{m,n}(u,x,y)\nonumber\\
&\leq \mathbb{P}\bigg(\max_{|s-t|\leq \delta}|\xi_u(s)-\xi_u(t)|>\epsilon\, \left| \,  X_1(s_u)=u-\frac{x}{u},\,X_2(t_u)=u-\frac{y}{u}\right.\bigg)\nonumber\\
&\leq \frac{1}{\epsilon}\mathbb{E}\left(\max_{|s-t|\leq \delta}|\xi_u(s)-\xi_u(t)|\, \left | \,  X_1(s_u)=u-\frac{x}{u},\, X_2(t_u)=u-\frac{y}{u}\right.\right)\nonumber\\
&=\frac{1}{\epsilon}\mathbb{E}_{\mathbb{P}_u}\bigg(\max_{|s-t|\leq \delta}|x(s)-x(t)|\bigg),
\end{align}
where $\mathbb{P}_u$  is the probability measure on $(C(\mathbb{S}),\mathcal{B}(C(\mathbb{S}))$ defined as
\begin{align*}
&\mathbb{P}_u(A):=\mathbb{P}\bigg(\xi_u(\cdot)\in A\ \left |\ X_1(s_u)=u-\frac{x}{u},X_2(t_u)=u-\frac{y}{u}\right.\bigg),
\end{align*}
for all $A \in \mathcal{B}(C(\mathbb{S}))$ and $x(\cdot)$ is the coordinate random element on $(C(\mathbb{S}),
\mathcal{B}(C(\mathbb{S})),$ $\mathbb{P}_u)$, i.e., $x(t,\omega)=\omega (t), \ \forall \omega \in C(\mathbb{S})$ and $t \in \mathbb{S}$.
Consider the canonical metric
\begin{align}
d_u(s,t):&= \left[\mathbb{E}_{\mathbb{P}_u} \big(|x(s)-x(t)|^2\big)\right]^{1/2}\nonumber\\
&=\left[\mathbb{E}\left(|\xi_u(s)-\xi_u(t)|^2\ \left | \  X_1(s_u)=u-\frac{x}{u},X_2(t_u)=u-\frac{y}{u}\right.\right)\right]^{1/2}. \nonumber
\end{align}
By \eqref{conditional variance 4}, for $u$ large enough and all $s_u$, $t_u$ such that $|\tau_u|\leq C_0\sqrt{\log u}/u$, we have
\begin{equation} \label{Eq:canmetric}
d_u(s,t)\leq 2\sqrt{c_1}|s-t|^{\alpha_1/2},
\end{equation}
which implies $\forall s\in \mathbb S$,
\begin{align*}
\big\{t\in \mathbb S\ |\ |t-s|\leq (\epsilon/2\sqrt{c_1})^{\frac{2}{\alpha_1}}\big\}\subseteq \{t\in 
\mathbb S\ |\ d_u(s,t)\leq \epsilon\}.
\end{align*}
Hence
\begin{equation}\label{Eq:N}
N_{d_u}(\mathbb{S},\epsilon)\leq  C_0 \epsilon^{-2N/\alpha_1},
\end{equation}
where $N_{d_u}(\mathbb{S},\epsilon)$ denotes the minimum number of $d_u$-balls with radius $\epsilon$
that are needed to cover $\mathbb{S}$.  By Dudley's Theorem (see, e.g., Theorem 1.3.3 in
\cite{Adler_Taylor_2007}) and (\ref{Eq:canmetric}),  we have
\begin{equation}
\label{upper bound of cmuc 2}
\mathbb{E}_{\mathbb{P}_u} \bigg(\max_{|s-t|\leq \delta}|x(s)-x(t)|\bigg)
\leq K\int_0^{2\sqrt{c_1}\delta^{\alpha_1/2}} \sqrt{\log N_{d_u}(\mathbb{S},\epsilon)}\, d\epsilon,
\end{equation}
where  $K< \infty$ is a constant (which does not depend on $\delta$) and, thanks to (\ref{Eq:N}),
the last integral goes to 0 as $\delta \rightarrow 0$ (or, equivalently, as $m\rightarrow \infty, n\rightarrow \infty$).
By \eqref{upper bound of cmuc 1} and \eqref{upper bound of cmuc 2}, we conclude that
\begin{equation}
\label{upper bound of cmuc 4}
\lim_{\substack {m\rightarrow \infty \\ n\rightarrow \infty }}\limsup_{u\rightarrow \infty} \max_{|\tau_u|\leq C_0\sqrt{\log u}|/u}|s_{m,n}(u,x,y)|=0.
\end{equation}
A similar argument shows that
\begin{equation}
\label{upper bound of cmuc 5}
\lim_{\substack {m\rightarrow \infty \\ n\rightarrow \infty }}\limsup_{u\rightarrow \infty} \max_{|\tau_u|\leq C_0\sqrt{\log u}|/u}|t_{m,n}(u,x,y)|=0.
\end{equation}
Since
\begin{align}
\label{limit of f(u,x,y)}
&|f(u,x,y)-f(x,y)|\leq
 |f_{m,n}(u,x,y)-f(x,y)|+|g_{m,n}(u,x,y)-f(x,y)|\nonumber\\
&\leq
 |f_{m,n}(u,x,y)-f(x,y)|+|f_{m,n}(u,x-\epsilon,y-\epsilon)-f(x-\epsilon,y-\epsilon)|\nonumber\\
&\qquad + |f(x-\epsilon,y-\epsilon)-f(x,y)|+|s_{m,n}(u,x,y)|+|t_{m,n}(u,x,y)|,
\end{align}
we combine \eqref{lower bound of cmuc 5}, \eqref{upper bound of cmuc 4} and \eqref{upper bound of cmuc 5} to obtain
\begin{align}
\label{limit of f(u,x,y) 1}
&\limsup_{u\rightarrow \infty} \max_{|\tau_u|\leq C_0\sqrt{\log u}|/u}|f(u,x,y)-f(x,y)| \leq  |f(x-\epsilon,y-\epsilon)-f(x,y)|\nonumber\\
&\qquad +
 \lim_{\substack {m\rightarrow \infty \\ n\rightarrow \infty }}\limsup_{u\rightarrow \infty} \max_{|\tau_u|\leq C_0\sqrt{\log u}/u}
 \Big(|f_{m,n}(u,x,y)-f(x,y)|  +|s_{m,n}(u,x,y)|\nonumber\\
 &\qquad+|t_{m,n}(u,x,y)| + |f_{m,n}(u,x-\epsilon,y-\epsilon)-f(x-\epsilon,y-\epsilon)|\Big)\nonumber\\
&= |f(x-\epsilon,y-\epsilon)-f(x,y)|.\nonumber
\end{align}
Since the last term $\rightarrow 0$ as $\epsilon \downarrow 0$, we have completed the
proof of the second part of the lemma.
\end{proof}

Now we are ready to prove the main lemmas in Section 3.

\begin{proof}[Proof of Lemma \ref{main lemma}]
Let $\phi(a,b)$ be the density of $(X_1(s_u),X_2(t_u))^T$, i.e.,
\begin{equation}
\phi(a,b)=\frac{1}{2\pi\sqrt{1-r^2(|\tau_u|)}}\exp\left\{-\frac{1}{2}\frac{a^2-2r(|\tau_u|)ab+b^2}{1-r^2(|\tau_u|)}\right\}.
\end{equation}
By conditioning and a change of variables, the LHS of \eqref{local double extremes asymptotics} becomes
\begin{align}
\label{main lemma 1}
&\mathbb{P}\bigg(\max_{s\in s_u+u^{-2/\alpha_1}\mathbb{S}} X_1(s)>u, \max_{t\in t_u+u^{-2/\alpha_2}\mathbb{T}} X_2(t)>u\bigg)\nonumber\\
&=\int_{\mathbb{R}^2} \mathbb{P}\bigg(\max_{s\in s_u+u^{-2/\alpha_1}\mathbb{S}} X_1(s)>u, 
\max_{t\in t_u+u^{-2/\alpha_2}\mathbb{T}} X_2(t)>u\ \left |\ X_1(s_u)=u-\frac{x}{u},\right.\nonumber\\
&\qquad \qquad \ \ X_2(t_u)=u-\frac{y}{u}\bigg)\phi\Big(u-\frac{x}{u},u-\frac{y}{u}\Big)u^{-2}dxdy\nonumber\\
&=\frac{1}{2\pi\sqrt{1-r^2(|\tau_u|)}u^2}\exp\left(-\frac{u^2}{1+r(|\tau_u|)}\right)\int_{\mathbb{R}^2} f(u,x,y)\tilde\phi(u,x,y)dxdy,
\end{align}
where $f(u,x,y)$ is defined in \eqref{f(u,x,y)} with $\xi_u(\cdot),\eta_u(\cdot)$ in  \eqref{tangent field}, and where
\begin{align*}
&\tilde \phi(u,x,y)\nonumber\\
:=&\exp\bigg\{-\frac{1}{2(1-r^2(|\tau_u|))}\Big(\frac{x^2+y^2}{u^2}-2(1-r(|\tau_u|))(x+y)-2r(|\tau_u|)\frac{xy}{u^2}\Big)\bigg\}.
\end{align*}
Since $\max_{|\tau_u|\leq C_0\sqrt{\log u}/u}|r(|\tau_u|)-\rho|\rightarrow 0$ as $u\rightarrow \infty$, it is easy to check that
\begin{equation}
\label{lim of tilde phi}
\max_{|\tau_u|\leq C_0\sqrt{\log u}/u}\left|\tilde{\phi}(u,x,y)-e^{\frac{x+y}{1+\rho}}\right|\rightarrow 0, \ \text{ as }\  u\rightarrow \infty.
\end{equation}
Recall $H_\alpha(\cdot)$ in \eqref{H_alpha(S,T)} and $f(x,y)$ in \eqref{f(x,y)}. Since $\xi(\cdot)$, $\eta(\cdot)$ are independent, and
\begin{align*}
\{\xi(t), t \in \R^N\} \stackrel{d}{=} \bigg\{(1+\rho)\Big[\chi_1\Big(\Big(\frac{\sqrt{c_1}}{1+\rho}\Big)^{\frac{2}{\alpha_1}}t\Big)-
\left|\Big(\frac{\sqrt{c_1}}{1+\rho}\Big)^{\frac{2}{\alpha_1}}t\right |^{\alpha_1}\Big], t\in \R^N\bigg\}, \\
\{ \eta(t), t \in \R^N\} \stackrel{d}{= } \bigg\{(1+\rho)\Big[\chi_2\Big(\Big(\frac{\sqrt{c_2}}{1+\rho}\Big)^{\frac{2}{\alpha_2}}t\Big)-
\left|\Big(\frac{\sqrt{c_2}}{1+\rho}\Big)^{\frac{2}{\alpha_2}}t\right |^{\alpha_2}\Big], t\in \R^N\bigg\},
\end{align*}
where $\stackrel{d}{=}$ means equality of all finite dimensional distributions, we have
\begin{align}
\label{main lemma 4}
&\int_{\mathbb{R}^2} f(x,y)e^{\frac{x+y}{1+\rho}}dxdy\nonumber\\
&=\int_{\mathbb{R}}e^{\frac{x}{1+\rho}}\mathbb{P}\Big(\max_{s\in \mathbb{S}}\xi(s)>x\Big)dx
\int_{\mathbb{R}}e^{\frac{y}{1+\rho}}\mathbb{P}\Big(\max_{t\in \mathbb{T}}\eta(t)>y\Big)dy\nonumber\\
&=(1+\rho)^2H_{\alpha_1}\bigg(\frac{c_1^{1/\alpha_1}\mathbb{S}}{(1+\rho)^{\frac{2}{\alpha_1}}}\bigg)
H_{\alpha_2}\bigg(\frac{c_2^{1/\alpha_2}\mathbb{T}}{(1+\rho)^{\frac{2}{\alpha_2}}}\bigg).
\end{align}
By \eqref{main lemma 1} and \eqref{main lemma 4}, to conclude the lemma, it suffices to prove
\begin{equation}
\label{main lemma 2}
\lim_{u\rightarrow \infty}\int_{\mathbb{R}^2}\max_{|\tau_u|\leq C_0\sqrt{\log u}/u}\left|f(u,x,y) \tilde \phi(u,x,y)
-f(x,y)e^{\frac{x+y}{1+\rho}}\right|\,dxdy=0.
\end{equation}
Firstly, applying Lemma \ref{uniformly conditional convergence} together with \eqref{lim of tilde phi}, 
we have
\begin{align}
\label{main lemma 3}
&\max_{|\tau_u|\leq C_0\sqrt{\log u}/u}\left|f(u,x,y)\tilde \phi(u,x,y)-f(x,y)e^{\frac{x+y}{1+\rho}}\right|
\rightarrow 0, \ \text{as}\ u\rightarrow \infty.
\end{align}
Secondly, as in \cite{Ladneva_Piterbarg_2000}, we can find an integrable dominating function 
$g\in L({\mathbb{R}^2})$ such that for $u$ large enough,
\begin{equation}
\max_{|\tau_u|\leq C_0\sqrt{\log u}/u}\left|f(u,x,y)\tilde \phi(u,x,y)-f(x,y)e^{\frac{x+y}{1+\rho}}\right|\leq g(x,y).
\end{equation}
Therefore, \eqref{main lemma 2} follows from the dominated convergence theorem. This finishes the proof. 
\end{proof}

\begin{proof}[Proof of Lemma \ref{double double local extremes lemma}]
We first claim that for any compact sets $\mathbb{S}$ and $\mathbb{T}$, the identity
\begin{equation}
\label{relation of H(S,T) and H(S)}
H_\alpha(\mathbb{S})+H_\alpha(\mathbb{T})-H_\alpha(\mathbb{S}\cup \mathbb{T})=H_\alpha(\mathbb{S},\mathbb{T}) 
\end{equation}
holds. Indeed, if we let $X=\sup_{t\in \mathbb{S}}(\chi(t)-|t|^\alpha)$ and $Y=\sup_{t\in \mathbb{T}}(\chi(t)-|t|^\alpha)$, then
\begin{align*}
&H_\alpha(\mathbb{S})+H_\alpha(\mathbb{T})-H_\alpha(\mathbb{S}\cup \mathbb{T})=\mathbb{E}\big(e^X\big)
+\mathbb{E}\big(e^Y\big)-\mathbb{E}\big(e^{\max(X,Y)}\big)\\
&=\mathbb{E}\big(e^X1_{\{X<Y\}}\big)+\mathbb{E}\big(e^Y1_{\{X\geq Y\}}\big)=\mathbb{E}\big(e^{\min(X,Y)}\big)=H_\alpha(\mathbb{S},\mathbb{T}).
\end{align*}
Now let $\mathbb{T}_1=[0,T]^N$, $\mathbb{T}_2=[\mathbf{m}T,(\mathbf{m}+1)T]$ and $\mathbb{T}_3=[\mathbf{n}T,(\mathbf{n}+1)T]$. 
Consider the events
\begin{align*}
&A=\left\{\max_{s\in s_u+u^{-2/\alpha_1}\mathbb{T}_1} X_1(s)>u\right\},\ \ B=\left\{\max_{s\in s_u+u^{-2/\alpha_1}\mathbb{T}_2} X_1(s)>u\right\},\\
&C=\left\{\max_{t\in t_u+u^{-2/\alpha_2}\mathbb{T}_1} X_2(t)>u\right\},\ \ D=\left\{\max_{t\in t_u+u^{-2/\alpha_2}\mathbb{T}_3} X_2(t)>u\right\}.
\end{align*}
It is easy to check that the LHS of \eqref{double double local extremes asymptotics} is equal to
\begin{align}
\label{double double local extremes asymptotics 1}
&\mathbb{P}(A\cap B\cap C\cap D)\nonumber\\
&=[\mathbb{P}(A\cap C)+\mathbb{P}( B\cap C)-\mathbb{P}((A\cup B)\cap C)]\nonumber\\
&+[\mathbb{P}(A\cap D)+\mathbb{P}( B\cap D)-\mathbb{P}((A\cup B)\cap D)]\nonumber\\
&-[\mathbb{P}(A\cap (C\cup D))+\mathbb{P}( B\cap (C\cup D))-\mathbb{P}((A\cup B)\cap (C\cup D))].
\end{align}
Let $R(u)=\frac{(1+\rho)^2}{2\pi \sqrt{1-\rho^2}}u^{-2}\exp\left(-\frac{u^2}{1+r(|\tau_u|)}\right)$ 
and $q_{\alpha,c}=\frac{(1+\rho)^{2/\alpha}}{c^{1/\alpha}}$. By Lemma \ref{main lemma}, we have
\begin{align*}
&\mathbb{P}(A\cap C)=R(u)H_{\alpha_1}\left(\frac{\mathbb{T}_1}{q_{\alpha_1,c_1}}\right)H_{\alpha_2}\left(\frac{\mathbb{T}_1}{q_{\alpha_2,c_2}}\right)(1+\gamma_1(u)),\\
&\mathbb{P}(B\cap C)=R(u)H_{\alpha_1}\left(\frac{\mathbb{T}_2}{q_{\alpha_1,c_1}}\right)H_{\alpha_2}\left(\frac{\mathbb{T}_1}{q_{\alpha_2,c_2}}\right)(1+\gamma_2(u)),\\
&\mathbb{P}((A\cup B)\cap C)=R(u)H_{\alpha_1}\left(\frac{\mathbb{T}_1\cup \mathbb{T}_2}{q_{\alpha_1,c_1}}\right)H_{\alpha_2}\left(\frac{\mathbb{T}_1}{q_{\alpha_2,c_2}}\right)(1+\gamma_3(u)),
\end{align*}
where, for $i = 1, 2, 3$,  $\gamma_i(u)\rightarrow 0$ uniformly w.r.t. $\tau_u$ satisfying $|\tau_u|\leq C_0\sqrt{\log u}/u$, as $u\rightarrow \infty$.
These, together with (\ref{relation of H(S,T) and H(S)}), imply
\begin{align}
\label{double double local extremes asymptotics 2}
&\mathbb{P}(A\cap C)+\mathbb{P}( B\cap C)-\mathbb{P}((A\cup B)\cap C)\nonumber\\
&=R(u)H_{\alpha_2}\left(\frac{\mathbb{T}_1}{q_{\alpha_2,c_2}}\right)H_{\alpha_1}\left(\frac{\mathbb{T}_1}{q_{\alpha_1,c_1}},
\frac{\mathbb{T}_2}{q_{\alpha_1,c_1}}\right)(1+o(1)).
\end{align}
Similarly, we have
\begin{align}
\label{double double local extremes asymptotics 3}
&\mathbb{P}(A\cap D)+\mathbb{P}( B\cap D)-\mathbb{P}((A\cup B)\cap D)\nonumber\\
&=R(u)H_{\alpha_2}\left(\frac{ \mathbb{T}_3 }{q_{\alpha_2,c_2}}\right)H_{\alpha_1}
\left(\frac{\mathbb{T}_1}{q_{\alpha_1,c_1}},\frac{\mathbb{T}_2}{q_{\alpha_1,c_1}}\right)(1+o(1))
\end{align}
and
\begin{align}
\label{double double local extremes asymptotics 4}
&\mathbb{P}(A\cap (C\cup D))+\mathbb{P}( B\cap (C\cup D))-\mathbb{P}((A\cup B)\cap (C\cup D))\nonumber\\
&=R(u)H_{\alpha_2}\left(\frac{\mathbb{T}_1\cup \mathbb{T}_3 }{q_{\alpha_2,c_2}}\right)H_{\alpha_1}\left(\frac{\mathbb{T}_1}{q_{\alpha_1,c_1}},\frac{\mathbb{T}_2}{q_{\alpha_1,c_1}}\right)(1+o(1)).
\end{align}
By \eqref{double double local extremes asymptotics 1} -- \eqref{double double local extremes asymptotics 4}, we have
\begin{align*}
&\mathbb{P}(A\cap B\cap C\cap D)\nonumber\\
&=R(u)H_{\alpha_1}\left(\frac{\mathbb{T}_1}{q_{\alpha_1,c_1}},\frac{\mathbb{T}_2}{q_{\alpha_1,c_1}}\right)
H_{\alpha_2,c_2}\left(\frac{\mathbb{T}_1}{q_{\alpha_2,c_2}},\frac{\mathbb{T}_3}{q_{\alpha_2,c_2}}\right)(1+o(1)),
\end{align*}
which concludes the lemma.
\end{proof}

\begin{proof}[Proof of Lemma \ref{lem_Riemann Sum}]
\label{proof of lemma_Riemann Sum}
Let $f(|t|)=\frac{1}{1+r(|t|)}$. Recall $\tau_{\mathbf{k}\mathbf{l}}$ defined in \eqref{def_tau_kl} 
and $|\tau_{\mathbf{k}\mathbf{l}}|\leq 2\delta(u)$, when $u$ is large. By Taylor's expansion,
\begin{align*}
f(|\tau_{\mathbf{k}\mathbf{l}}|)=f(0)+\frac{1}{2}f''(0)|\tau_{\mathbf{k}\mathbf{l}}|^2(1+\gamma_{\mathbf{k}\mathbf{l}}(u)),
\end{align*}
where $f(0)=\frac{1}{1+ \rho}$, $f''(0)=\frac{-r''(0)}{(1+\rho)^2}$ and, as $u\rightarrow \infty$, $\gamma_{\mathbf{k}\mathbf{l}}(u)$ 
converges to zero uniformly  w.r.t. all $(\mathbf{k},\mathbf{l})\in \mathcal{C}$.
Therefore, for any $\epsilon >0$,  we have
\begin{equation}
\label{bounds for h(u)}
\sum_{(\mathbf{k},\mathbf{l})\in \mathcal{C}} e^{-\frac{1}{2}f''(0)(1+\epsilon)u^2|\tau_{\mathbf{k}\mathbf{l}}|^2}\leq h(u)\leq \sum_{(\mathbf{k},\mathbf{l})\in \mathcal{C}} e^{-\frac{1}{2}f''(0)(1-\epsilon)u^2|\tau_{\mathbf{k}\mathbf{l}}|^2}
\end{equation}
when $u$ is large enough. For $ a>0$, let
\begin{equation}
h(u,a):=\sum_{(\mathbf{k},\mathbf{l})\in \mathcal{C}} e^{-au^2|\tau_{\mathbf{k}\mathbf{l}}|^2}.
\end{equation}
In order to prove (\ref{h(u)}), it suffices to prove that
\begin{equation}
\label{h(u,a)}
\lim_{u\rightarrow \infty}u^Nd_1^N(u)d_2^N(u)h(u,a)=\left(\frac{\pi}{a}\right)^\frac{N}{2}mes_N(A_1\cap A_2).
\end{equation}
To this end, we write
\begin{align}
\label{intuition}
&u^Nd_1^N(u)d_2^N(u)h(u,a)\nonumber\\
&=\frac{1}{u^N}\sum_{(\mathbf{k},\mathbf{l})\in \mathcal{C}} e^{-a\sum_{j=1}^N({l_jud_2(u)-k_jud_1(u)})^2}\cdot(ud_1(u))^N(ud_2(u))^N.
\end{align}
Let 
\begin{align*}
\label{bounds for h(u,a)}
&p(u):=\frac{1}{u^N}\sum_{(\mathbf{k},\mathbf{l})\in \mathcal{C}} \min_{(s,t)\in u\Delta_\mathbf{k}^{(1)}\times u\Delta_\mathbf{l}^{(2)}}e^{-a|t-s|^2}\cdot(ud_1(u))^N(ud_2(u))^N,\nonumber\\
&q(u):=\frac{1}{u^N}\sum_{(\mathbf{k},\mathbf{l})\in \mathcal{C}} \max_{(s,t)\in u\Delta_\mathbf{k}^{(1)}\times u\Delta_\mathbf{l}^{(2)}}e^{-a|t-s|^2}\cdot(ud_1(u))^N(ud_2(u)^N.
\end{align*}
It follows from (\ref{intuition}) that
\begin{equation}\label{Eq:pq1}
p(u)\leq u^Nd_1^N(u)d_2^N(u)h(u,a)\leq q(u),
\end{equation}
and
\begin{equation} \label{Eq:pq2}
p(u)\leq \frac{1}{u^N}\int_{\substack{s \in uA_1,t\in uA_2\\ |t-s|\leq C\sqrt{\log u}}}e^{-a|t-s|^2}dtds\leq q(u).
\end{equation}
Observe that
\begin{align} \label{Eq:pq3}
&\frac{1}{u^N}\iint_{\substack{s\in uA_1,t\in uA_2\\ |t-s|\leq C\sqrt{\log u}}}e^{-a|t-s|^2}dtds =\frac{1}{u^N}\iint_{\substack{y \in uA_1,x+y\in uA_2\\ |x|\leq C\sqrt{\log u}}}e^{-a|x|^2}dxdy\nonumber\\
&=\frac{1}{u^N}\int_{|x|\leq C\sqrt{\log u}}e^{-a|x|^2}dx\int_{\mathbb{R}^N}1_{\{y\in uA_1\cap (uA_2-x)\}}dy\nonumber\\
&=\int_{|x|\leq C\sqrt{\log u}}e^{-a|x|^2}dx\int_{\mathbb{R}^N}1_{\{z\in A_1\cap (A_2-x/u)\}}dz\nonumber\\
&\rightarrow \ mes_N(A_1\cap A_2)\int_{\mathbb{R}^N} e^{-a|x|^2}dx =\left(\frac{\pi}{a}\right)^\frac{N}{2}mes_N(A_1\cap A_2),
\end{align}
as $u\rightarrow \infty$, where the convergence holds by the dominated convergence theorem. Indeed, 
$\int_{\mathbb{R}^N}1_{\{z\in A_1\cap (A_2-x/u)\}}dz$ is bounded by $\max_{|\epsilon|<1} mes_N(A_1\cap (A_2-\epsilon))$ 
uniformly for $|x|\leq C\sqrt{\log u}$ when $u$ is large enough.

It follows from (\ref{Eq:pq1})--(\ref{Eq:pq1}) that, for concluding \eqref{h(u,a)}, 
it remains to verify  
\begin{equation}
\label{D(u) def}
D(u):=q(u)-p(u)\rightarrow 0, \ \text{as}\ u\rightarrow \infty.
\end{equation}

Define
\begin{align}
\mathcal{\hat D}:=\Big\{(s,t)\in A_1\times A_2:\,  |t-s|\leq \delta(u)+\sqrt{N}d_1(u)+\sqrt{N}d_2(u)\Big\}.
\end{align}
By the definition of $\mathcal{C}$ in \eqref{mathcal C}, we see that
$
\mathcal{D}\subseteq \bigcup_{(\mathbf{k},\mathbf{l})\in \mathcal{C}}\Delta_{\mathbf{k}}^{(1)}\times \Delta_{\mathbf{l}}^{(2)}\subseteq  \mathcal{\hat D}.
$
Since $d_1(u)=o(\delta(u))$ and $d_2(u)=o(\delta(u))$ as $u\rightarrow \infty$, the set 
$\mathcal{\hat D}$ is a subset of $\mathcal{\tilde D}:=\{(s,t)\in A_1\times A_2:\,  |t-s|\leq 2\delta(u)\}$ when $u$ is large.

Write $D(u)$ in (\ref{D(u) def}) as a sum over $(\mathbf{k},\mathbf{l})\in \mathcal{C}$. To estimate the cardinality of $ \mathcal{C}$, 
we notice that  
\begin{align}
\label{upper bound of mes(tilde D)}
 mes_{2N}(\mathcal{\tilde D})&= \iint_{s\in A_1, t\in A_2}1_{\{|t-s|\leq 2\delta(u)\}}dsdt \\
 &=\int_{|x|\leq 2\delta(u)}\int_{y \in A_1\cap (A_2-x)}dydx
\leq K\,\delta(u)^N,
\end{align}
for all $u$ large enough, where $K=2^{N+1}\pi^{N/2}\Gamma^{-1}(N/2)\max_{| \epsilon| \leq 1} mes_N(A_1\cap (A_2-\epsilon))$.
Hence, for large $u$, the number of summands in (\ref{intuition}) is bounded by
\begin{equation}
\label{number of sums in D(u)}
\#\{(\mathbf{k},\mathbf{l})\ |\ (\mathbf{k},\mathbf{l})\in \mathcal{C}\}\leq
\frac{mes_{2N}(\mathcal{\tilde D})} {mes_{2N}(\Delta_{\mathbf{k}}^{(1)}\times 
\Delta_{\mathbf{l}}^{(2)})}\leq \frac{K\,\delta(u)^N} {d_1^N(u)d_2^N(u)}.
\end{equation}

Next, by applying the inequality $e^{-x}-e^{-y}\leq y-x$ for $y\geq x>0$ to each summand in $D(u)$, we obtain
\begin{align}
\label{bound for element in D(u)_0}
&\max_{(s,t)\in u\Delta_\mathbf{k}^{(1)}\times u\Delta_\mathbf{l}^{(2)}}e^{-a|t-s|^2}-\min_{(s,t)\in u\Delta_\mathbf{k}^{(1)}\times u\Delta_\mathbf{l}^{(2)}}e^{-a|t-s|^2}\nonumber\\
&\leq a\left(\max_{(s,t)\in u\Delta_\mathbf{k}^{(1)}\times u\Delta_\mathbf{l}^{(2)}}|t-s|^2-\min_{(s,t)\in u\Delta_\mathbf{k}^{(1)}\times u\Delta_\mathbf{l}^{(2)}}|t-s|^2\right)\nonumber\\
&=a \max \big(|t-s|+|t_1-s_1|)(|t-s|-|t_1-s_1|\big),
\end{align}
where the last maximum is taken over $(s,t,s_1,t_1) \in u\Delta_\mathbf{k}^{(1)}\times u\Delta_\mathbf{l}^{(2)}\times u\Delta_\mathbf{k}^{(1)}\times u\Delta_\mathbf{l}^{(2)}$

Since $|t-s|\leq 2\delta(u)$ for all $(t,s)\in u\Delta_\mathbf{k}^{(1)}\times u\Delta_\mathbf{l}^{(2)}$ when $u$ is large, the inequality
$\big||t-s|-|t_1-s_1|\big|\leq |t-t_1|+|s-s_1|$ implies that \eqref{bound for element in D(u)_0} is at most
\begin{align}
\label{bound for element in D(u)}
4a\sqrt{N}u^2\delta(u)\big(d_1(u)+d_2(u)\big)
\end{align}
when $u$ is large enough. By \eqref{bound for element in D(u)} and \eqref{number of sums in D(u)}, we can verify that  
\begin{align*}
D(u)&\leq  \frac{1}{u^N}\frac{K(\delta(u))^N}{d_1^N(u)d_2^N(u)}4a\sqrt{N}u^2\delta(u)\big(d_1(u)+d_2(u)\big)\big(ud_1(u)\big)^N\big(ud_2(u)\big)^N\nonumber\\
&\leq C_0 (\log u)^{\frac{N+1}{2}}\big(u^{1-\frac{2}{\alpha_1}}+u^{1-\frac{2}{\alpha_2}}\big)\, \rightarrow 0,\ \text{ as }\ u\rightarrow \infty.
\end{align*}
Therefore  \eqref{h(u,a)} holds. Similarly, we can check that the same statement holds while changing the set $\mathcal{C}$ to $\mathcal{C}^\circ$.
\end{proof}

\begin{proof}[Proof of Lemma \ref{lem_mathcal_H(m)}]
Inequality \eqref{mathcal_H(0) bounds} holds immediately by Lemma $6.2$ in \cite{Piterbarg_1996}. Hence we only 
consider the case when $\mathbf{m}\neq \mathbf{0}$. Suppose that $\{X(t),t\in \mathbb{R}^N\}$ is a real valued 
continuous Gaussian process with $\E[X(t)]=0$ and covariance function $r(t)$ satisfying $r(t)=1-|t|^\alpha+o(|t|^\alpha)$ 
for a constant $\alpha \in (0,2)$. Applying Lemma $6.1$ in \cite{Piterbarg_1996}, we see that for any $S>0$,  
\begin{align}
\label{mathcal_H_alpha(m) with extreme prob}
&\mathbb{P}\bigg(\max_{t\in u^{-2/\alpha}[0,S]^N}X(t)>u,\, \max_{t\in u^{-2/\alpha}[\mathbf{m}S,(\mathbf{m}+1)S]}X(t)>u\bigg)\nonumber\\
&=\mathbb{P}\bigg(\max_{t\in u^{-2/\alpha}[0,S]^N}X(t)>u\bigg)+\mathbb{P}\bigg(\max_{t\in u^{-2/\alpha}[\mathbf{m}S,(\mathbf{m}+1)S]}X(t)>u\bigg)\nonumber\\
&\qquad  \ \ -\mathbb{P}\bigg(\max_{t\in u^{-2/\alpha}([0,S]^N\cup [\mathbf{m}S,(\mathbf{m}+1)S])}X(t)>u)\bigg)\nonumber\\
&=\Big(H_\alpha([0,S]^N)+H_\alpha([\mathbf{m}S,(\mathbf{m}+1)S])-H_\alpha([0,S]^N\cup [\mathbf{m}S,(\mathbf{m}+1)S])\Big)\nonumber\\
&\qquad \ \ \times \frac{1}{\sqrt{2\pi}u}e^{-\frac{1}{2}u^2}(1+o(1))\nonumber\\
&=H_\alpha([0,S]^N,[\mathbf{m}S,(\mathbf{m}+1)S])\frac{1}{\sqrt{2\pi}u}e^{-\frac{1}{2}u^2}(1+o(1)),\ \text{ as }\ u\rightarrow \infty,
\end{align}
where the last equality holds thanks to \eqref{relation of H(S,T) and H(S)}.

On the other hand, by applying Lemma $6.3$ in \cite{Piterbarg_1996} and the inequality
$\inf_{s\in [0,1]^N,t\in [\mathbf{m},\mathbf{m}+1]}|s-t|\geq |m_{i_0}|-1$ (recall that $i_0$ is defined in Lemma \ref{lem_mathcal_H(m)}),  we have
\begin{align}
\label{Lemma 6.3 in Piterbarg}
&\mathbb{P}\bigg(\max_{t\in u^{-2/\alpha}[0,S]^N}X(t)>u,\, \max_{t\in u^{-2/\alpha}[\mathbf{m}S,(\mathbf{m}+1)S]}X(t)>u\bigg)\nonumber\\
&\leq C_0 S^{2N}\frac{1}{\sqrt{2\pi}u}e^{-\frac{1}{2}u^2}\exp\left(-\frac{1}{8}(|m_{i_0}|-1)^\alpha S^\alpha\right)
\end{align}
for all $u$ large enough. It follows from \eqref{mathcal_H_alpha(m) with extreme prob} and \eqref{Lemma 6.3 in Piterbarg} that
\begin{equation}
H_\alpha([0,S]^N,[\mathbf{m}S,(\mathbf{m}+1)S])\leq C_0S^{2N}\exp\left(-\frac{1}{8}(|m_{i_0}|-1)^\alpha S^\alpha\right),
\end{equation}
which implies \eqref{mathcal_H(m) bounds} by letting $S=\frac{c^{1/\alpha}T}{(1+\rho)^{2/\alpha}}$.

When $|m_{i_0}|=1$, the above upper bound is not sharp. Instead, we derive \eqref{mathcal_H(1) bounds} 
in Lemma \ref{lem_mathcal_H(m)} as follows. For concreteness, suppose that $i_0=N$ and $m_N=1$. By applying Lemmas 
$6.1$ - $6.3$ in  \cite{Piterbarg_1996}, we have
\begin{align}
&\mathbb{P}\bigg(\max_{t\in u^{-2/\alpha}[0,S]^N}X(t)>u,\, \max_{t\in u^{-2/\alpha}[\mathbf{m}S,(\mathbf{m}+1)S]}X(t)>u\bigg)\nonumber\\
&\leq \mathbb{P}\bigg(\max_{t\in u^{-2/\alpha}(\prod_{j=1}^{N-1}[m_jS,(m_j+1)S]\times[S,S+\sqrt{S}])}X(t)>u\bigg)\nonumber\\
&+\mathbb{P}\bigg(\max_{t\in u^{-2/\alpha}[0,S]^N}X(t)>u,\, \max_{t\in u^{-2/\alpha}(\prod_{j=1}^{N-1}[m_jS,(m_j+1)S]\times[S+\sqrt{S},2S+\sqrt{S}])}X(t)>u\bigg)\nonumber\\
&\leq C_0 S^{N-\frac{1}{2}}\frac{1}{\sqrt{2\pi}u}e^{-\frac{1}{2}u^2}+C_0 S^{2N}\frac{1}{\sqrt{2\pi}u}e^{-\frac{1}{2}u^2}e^{-\frac{1}{8}S^{\alpha/2}}\nonumber\\
&\leq C_0 S^{N-\frac{1}{2}}\frac{1}{\sqrt{2\pi}u}e^{-\frac{1}{2}u^2} 
\end{align}
for $u$ and $S$ large. Hence, when $|m_{i_0}|=1$,  we have
\begin{equation}
H_\alpha([0,S]^N,[\mathbf{m}S,(\mathbf{m}+1)S])\leq C_0 S^{N-\frac{1}{2}}
\end{equation}
for large $S$. This  implies \eqref{mathcal_H(1) bounds} by letting $S=\frac{c^{1/\alpha}T}{(1+\rho)^{2/\alpha}}$.

Notice that
\begin{equation}
\#\{m\in \mathbb{Z}^N\ |\ \max_{1\leq i\leq N}|m_i|=k\}=(2k+1)^N-(2k-1)^N,\ k=1,2,....
\end{equation}
By \eqref{mathcal_H(1) bounds}, \eqref{mathcal_H(m) bounds} and the fact that $\int_T^\infty x^{N-1} e^{-a x^\alpha}dx
\sim \frac {1} {a\alpha} T^{N-\alpha} e^{-aT^\alpha}$ as $T \to \infty$, we have
\begin{align*}
&\sum_{\mathbf{m}\neq \mathbf{0}}\mathcal{H}_{\alpha,c}(\mathbf{m})=\sum_{k=1}^\infty\sum_{|m_{i_0}|=k}\mathcal{H}_{\alpha,c}(\mathbf{m})\\&\leq C_0 (3^N-1)T^{N-\frac{1}{2}}+C_0 \sum_{k=2}^\infty [(2k+1)^N-(2k-1)^N]T^{2N}e^{-\frac{c}{8(1+\rho)^2}(k-1)^\alpha T^\alpha}\\
&\leq  C_0 (3^N-1)T^{N-\frac{1}{2}}+C_0 T^{2N}\int_{1}^\infty x^{N-1}e^{-\frac{c}{8(1+\rho)^2}x^\alpha T^\alpha}dx \leq  C_0  T^{N-\frac{1}{2}}
\end{align*}
for $T$ large enough. This completes the proof of Lemma \ref{lem_mathcal_H(m)}.
\end{proof}

\begin{proof} [Proof of Lemma \ref{lem_Riemann Sum-2}] 
The proof is similar to that of Lemma \ref{lem_Riemann Sum}.
Indeed, we only need to modify \eqref{intuition} and \eqref{upper bound of mes(tilde D)} in the proof of Lemma
\ref{lem_Riemann Sum}. For any $ y=(y_1,...,y_N)\in \mathbb{R}^N$ and $1\leq i\leq j\leq N$, let  $y_{i:j} =(y_i,...,y_j)$.
On one hand, with a different scaling, $h(u,a)$ in \eqref{intuition} has the following asymptotics:
\begin{align}
\label{intuition_2}
&u^{2N-M}d_1^N(u)d_2^N(u)h(u,a) \approx\frac{1}{u^M}\iint_{\substack{y \in uA_1,x+y\in uA_2\\ |x|\leq C\sqrt{\log u}}}e^{-a|x|^2}dxdy\nonumber\\
&= \frac{1}{u^M}\int_{|x|\leq C\sqrt{\log u}}e^{-a|x|^2}\bigg(\int_{\mathbb{R}^M}1_{\{y_{1:M}\in uA_{1,M}\cap (uA_{2,M}-x_{1:M})\}}dy_{1:M}\nonumber\\
&\qquad \times\prod_{j=M+1}^N\int_{\mathbb{R}}1_{\{y_j\in [uS_j,uT_j]\cap[uT_j-x_j,uR_j-x_j]\}}dy_j\bigg)\, dx\nonumber\\
&=\int_{|x|\leq C\sqrt{\log u}}e^{-a|x|^2}\prod_{j=M+1}^N x_j 1_{\big\{x_j>0\big\}}\bigg(\int_{\mathbb{R}^M}
1_{\big\{z_{1:M}\in A_{1,M}\cap (A_{2,M}-x_{1:M}/u)\big\}}dz_{1:M}\bigg)\,dx\nonumber\\
& \rightarrow \ mes_M(A_{1,M}\cap A_{2,M})\int_{\mathbb{R}^M} e^{-a|x_{1:M}|^2}dx_{1:M}\prod_{j=M+1}^N
\int_0^\infty x_je^{-ax_j^2}dx_j\nonumber\\
&=2^{M-N}\pi^{M/2}a^{M/2-N}mes_M(A_{1,M}\cap A_{2,M}),
\end{align}
as $u\rightarrow \infty$. 
On the other hand, when $u$ is large enough,  $mes_{2N}( \mathcal{\tilde D})$ defined in \eqref{upper bound of mes(tilde D)}
can be bounded above by
\begin{align}
\label{upper bound of mes(tilde D)_2}
&mes_{2N}( \mathcal{\tilde D})=\iint_{s\in A_1, t\in A_2}1_{\{|t-s|\leq 2\delta(u)\}}dsdt\nonumber\\
&=\int_{|x|\leq 2\delta(u)}\Big(\int_{y_{1:M} \in A_{1,M}\cap (A_{2,M}-x_{1:M})}dy_{1:M}\Big)\prod_{j=M+1}^Nx_j1_{\{x_j>0\}}dx\nonumber\\
&=\delta(u)^{2N-M}\int_{|z|\leq 2}\Big(\int_{y_{1:M} \in A_{1,M}\cap (A_{2,M}-z_{1:M}\delta(u))}dy_{1:M}\Big)\prod_{j=M+1}^Nz_j1_{\{z_j>0\}}dz\nonumber\\
&\leq K\, \delta(u)^{2N-M},
\end{align}
where $K=\max_{| \epsilon| \leq 1} mes_M(A_{1,M}\cap (A_{2,M}-\epsilon))\int_{|z|\leq 2}\prod_{j=M+1}^Nz_j1_{\{z_j>0\}}dz$.

By \eqref{intuition_2} and \eqref{upper bound of mes(tilde D)_2}, \eqref{h(u)_mes zero} can be obtained through
the same argument in the proof of Lemma \ref{lem_Riemann Sum}. We omit the details.
\end{proof}

We end this section with the proof of Lemma \ref{lem_uniform convergence of f.d.d.}.
\begin{proof}[Proof of Lemma \ref{lem_uniform convergence of f.d.d.}]
Let $f_{u,\tau_u}(\cdot)$ and $f(\cdot)$ be the density function of $X(u,\tau_u)$ and $X$, respectively. It suffices to prove that
for all $ x\in \mathbb{R}^N$,
\begin{align}
\label{uniform covergence of F_u}
\int_{\{y\leq x\}}f(y)\max_{\tau_u}\bigg|\frac{f_{u,\tau_u(y)}}{f(y)}-1\bigg|dy\rightarrow 0, \ \text{as}\ u\rightarrow \infty,
\end{align}
where $\{y\leq x\}= \prod_{i=1}^N(-\infty, x_i]$. 

First, we will find an upper bound for $\max_{\tau_u}|f_{u,\tau_u}(y)/f(y)-1|$. 
For any $ \epsilon>0$, define
\begin{align*}
&\Gamma(u,\tau_u)=(\gamma_{ij}(u,\tau_u))_{i,j=1,...n}:=\frac{1}{\epsilon}(\Sigma(u,\tau_u)- \Sigma)\\
&e(u,\tau_u)=(e_{i}(u,\tau_u))_{i=1,...n}:=\frac{1}{\epsilon}(\mu(u,\tau_u)-\mu).
\end{align*}
By Assumption \eqref{mean_var_unif_conver}, there exists a constant $U>0$ such that for all $u>U$,
\begin{align*}
\max_{\tau_u}|\mu_j(u,\tau_u)-\mu_j|<\epsilon,\ \ \ \max_{\tau_u}|\sigma_{ij}(u,\tau_u)-\sigma_{ij}|<\epsilon, \ \  i,j=1,\dots, n,
\end{align*}
which implies $|\gamma_{ij}(u,\tau_u)|\leq 1$ and $|e_{i}(u,\tau_u)|\leq 1$ for $u>U$.

Let $\Sigma^{-1}=(v_{ij})_{i,j=1,...,n}$ be the inverse of $\Sigma$. When $\epsilon$ is small, the determinant of
$\Sigma(u,\tau_u)$ satisfies
\begin{align*}
|\Sigma(u,\tau_u)|=|\Sigma +\epsilon \Gamma(u,\tau_u)|=|\Sigma|(1+\epsilon \hbox{tr}(\Sigma^{-1}
\Gamma(u,\tau_u))+O(\epsilon^2)),
\end{align*}
where $O(\epsilon^2)/\epsilon^2$ is uniformly bounded w.r.t. $\tau_u$ for large $u$
(see, e.g., \cite{Magnus_Heudecker_2007}, p. $169$). Hence, when $\epsilon$ is small enough, we have
\begin{align}
\label{ratio of two var det}
\left|\frac{|\Sigma(u,\tau_u)|}{|\Sigma|}-1\right|\leq 2\epsilon |\hbox{tr}(\Sigma^{-1}\Gamma(u,\tau_u))|\leq 2\epsilon \sum_{i,j}|v_{ij}|.
\end{align}
Since  $|\gamma_{ij}(u,\tau_u)|\leq 1,\  \forall i,j=1,...,n,\ \forall \tau_u$ for large $u$,
as $\epsilon \rightarrow 0$, the inverse of $\Sigma(u,\tau_u)$  can be written as
\begin{align*}
\Sigma(u,\tau_u)^{-1}=\Sigma^{-1}-\epsilon \Sigma^{-1}\Gamma(u,\tau_u)\Sigma^{-1}+O(\epsilon^2),
\end{align*}
where $O(\epsilon^2)/\epsilon^2$ is a matrix whose entries are uniformly bounded and independent of $\tau_u$ for large $u$
 (see, e.g., \cite{Meyer_2000}, p. $618$). Hence,
\begin{align*}
d_{u,\tau_u}(y):=&-\frac{1}{2}\Big[(y-\mu(u,\tau_u))^T\Sigma^{-1}(u,\tau_u)(y-\mu(u,\tau_u))- (y-\mu)^T\Sigma^{-1}(y-\mu)\Big]\nonumber\\
= &-\frac{1}{2}(y-\mu)^T\big(-\epsilon \Sigma^{-1}\Gamma(u,\tau_u)\Sigma^{-1}+O(\epsilon^2)\big)(y-\mu)\nonumber\\
&+\epsilon e^T(u,\tau_u)\big(\Sigma^{-1}-\epsilon \Sigma^{-1}\Gamma(u,\tau_u)\Sigma^{-1}+O(\epsilon^2)\big)(y-\mu)\nonumber\\
&-\frac{1}{2}\epsilon^2e^T(u,\tau_u)\big(\Sigma^{-1}-\epsilon \Sigma^{-1}\Gamma(u,\tau_u)\Sigma^{-1}+O(\epsilon^2)\big)e(u,\tau_u). 
\end{align*}
Since $|\gamma_{ij}(u,\tau_u)|$ and $|e_{i}(u,\tau_u)|$ are uniformly bounded by $1$ w.r.t. $\tau_u$ for all $u>U$, 
we derive that for any $y\in \mathbb{R}^N$,
\begin{align}
\label{limit of d_u,tau_u(y)}
\max_{\tau_u}|d_{u,\tau_u}(y)|\rightarrow 0,\ \text{as}\ u\rightarrow \infty.
\end{align}
By \eqref{ratio of two var det} and \eqref{limit of d_u,tau_u(y)}, for $y\in \mathbb{R}^N$,
\begin{align}
\label{pointwise conv of integrand}
\max_{\tau_u}\left|\frac{f_{u,\tau_u(y)}}{f(y)}-1\right|=\max_{\tau_u}\Big | e^{d_{u,\tau_u}(y)}
\frac{|\Sigma(u,\tau_u)|^{-1/2}}{|\Sigma|^{-1/2}}-1\Big|\rightarrow 0,\ \text{as} \ u\rightarrow \infty.
\end{align}
If we could further find an integrable function $g(y)$ on $\mathbb{R}^N$,
\begin{align}
\label{dominating function g}
f(y)\max_{\tau_u}\left|\frac{f_{u,\tau_u}(y)}{f(y)}-1\right|\leq g(y),
\end{align}
then \eqref{uniform covergence of F_u} holds by the dominated convergence theorem.

Given a constant $C_0$, let 
$
A_I:=\{(a_{ij})_{i,j=1}^n\in  \mathbb{R}^{N\times N}\ |\ \max_{i,j} |a_{i,j}|\leq C_0\},\ 
b_I:=\{(b_i)_{i=1}^n\in \mathbb{R}^N\ | \ \max_{i} |b_i|\leq C_0\}
$. Then there exist constants $C_2,C_3$, such that
\begin{align*}
|x^TAx|\leq C_2 x^Tx,\ \ \
|b^Tx|\leq C_3+x^Tx, \ \ \forall x\in \mathbb{R}^N, \forall A\in A_I, \forall b\in b_I.
\end{align*}
Hence, there exists a constant $C_4>0$ such that
\begin{align}
\label{ratio of two exp}
|d_{u,\tau_u}(y)|\leq& C_4\epsilon(y-\mu)^T(y-\mu)+ C_4\epsilon ,
\end{align}
By \eqref{ratio of two var det} and \eqref{ratio of two exp}, for small $\epsilon$ and large $u$,
there exists a constant $K$ such that
\begin{align*}
&\max_{\tau_u}\left|\frac{f_{u,\tau_u(y)}}{f(y)}-1\right|
\leq Ke^{C_4\epsilon(y-\mu)^T(y-\mu)}+1
\end{align*}
On the other hand, for all $y\in \mathbb{R}^N$,
\begin{align*}
f(y)\leq (2\pi)^{-n/2}|\Sigma|^{-1/2}e^{-\frac{\lambda}{2}(y-\mu)^T(y-\mu)},
\end{align*}
where $\lambda$ is the minimum eigenvalue of $\Sigma^{-1}$. If we choose $\epsilon<\frac{\lambda}{2C_4}$ and define
\begin{align*}
g(y):=(2\pi)^{-n/2}|\Sigma|^{-1/2}e^{-\frac{\lambda}{2}(y-\mu)^T(y-\mu)}(Ke^{C_4\epsilon(y-\mu)^T(y-\mu)}+1),
\end{align*}
then \eqref{dominating function g} holds and hence we have completed the proof.
\end{proof}

\bibliographystyle{amsalpha}



\end{document}